\documentclass[12pt]{amsart}
\usepackage{amstext,amsfonts,amssymb,amscd,amsbsy,amsmath,verbatim, mathrsfs, fullpage}
\usepackage[alphabetic,abbrev,lite,backrefs]{amsrefs} % for bibliography 
\usepackage{ifthen,tikz}
\usepackage{color}
\usepackage{amsthm}
\usepackage{latexsym}
\usepackage[all]{xy}
\usepackage{enumerate}
\usepackage{mathtools}

\numberwithin{equation}{section}

\newtheorem{lemma}{Lemma}[section]

\newtheorem{prop}[lemma]{Proposition}
\newtheorem{cor}[lemma]{Corollary}
\newtheorem{conj}[lemma]{Conjecture}

\newtheorem{claim*}{Claim}
\newtheorem{thm}[lemma]{Theorem}
\newtheorem{notation}[lemma]{Notation}

\newtheorem{question}[lemma]{Question}

\theoremstyle{definition}
\newtheorem{defn}[lemma]{Definition}
\newtheorem{example}[lemma]{Example}

\newtheorem{conventions}[lemma]{Conventions}

\theoremstyle{remark}
\newtheorem{remark}[lemma]{Remark}
\newtheorem{remarks}[lemma]{Remarks}
\newtheorem{rem}[lemma]{Remark}

\newcommand{\Tate}{{\mathbf{T}}}

\newcommand{\cC}{\mathcal{C}}

\newcommand{\m}{\mathfrak m}
\newcommand{\PP}{\mathbb P}

\newcommand{\bA}{\mathbb A}
\newcommand{\A}{\bA}

\newcommand{\ZZ}{\mathbb Z}

\newcommand{\Spec}{\operatorname{Spec}}

\newcommand{\im}{\operatorname{im}}

\newcommand{\id}{\operatorname{id}}

\newcommand{\Tor}{\operatorname{Tor}}

\newcommand{\Hom}{\operatorname{Hom}} %done

\newcommand{\cO}{{\mathcal O}}

\newcommand{\F}{\FF}
\newcommand{\CC}{\mathbb{C}}

\newcommand{\defi}[1]{\textsf{#1}} % for defined terms

\newcommand{\beq}{\begin{displaymath}}
\newcommand{\eeq}{\end{displaymath}}

\def\reg{\operatorname{reg}}

\def\nc{\newcommand}
\def\on{\operatorname}
\nc{\Q}{\mathbb{Q}}
\nc{\RR}{\mathbf{R}}
\nc{\LL}{\mathbf{L}}
\nc{\xra}{\xrightarrow}
\nc{\xla}{\xleftarrow}
\def\a{\alpha}
\def\om{\omega}

\def\DM{\operatorname{DM}}
\def\Coh{\operatorname{Coh}}

\def\th{\on{th}}
\def\F{\mathcal{F}}
\def\coker{\on{coker}}

\def\l{\ell}

\nc{\into}{\hookrightarrow}
\nc{\onto}{\twoheadrightarrow}
\nc{\OO}{\mathcal{O}}
\nc{\Z}{\mathbb{Z}}
\nc{\cA}{\mathcal{A}}
\nc{\w}{\widehat}
\nc{\End}{\on{End}}
\nc{\res}{\frac{1}{x_0x_1}}
\nc{\tF}{\widetilde{F}}
\nc{\tG}{\widetilde{G}}
\nc{\tf}{\widetilde{f}}
\nc{\Com}{\on{Com}}

\nc{\G}{\mathbb{G}}
\nc{\cG}{\mathcal{G}}
\nc{\cE}{\mathcal E}
\nc{\cF}{\mathcal F}
\nc{\cR}{\mathcal R}
\nc{\cD}{\mathcal D}
\nc{\cB}{\mathcal B}
\nc{\cT}{\mathcal T}
\nc{\cL}{\mathcal L}

\nc{\bM}{\mathbf M}
\nc{\bN}{\mathbf N}
\nc{\U}{\mathbf U}
\nc{\BM}{\mathbf B \mathbf M}
\nc{\Dsg}{\on{D}_{\on{sg}}}
\nc{\fC}{\mathcal{C}}
\nc{\fG}{\mathcal{G}}
\nc{\N}{\mathbb{N}}

\nc{\del}{\partial}
\nc{\cone}{\on{cone}}
\nc{\D}{\on{D}_{\on{diff}}}
\nc{\DMb}{\on{D}^b_{\DM}}
\nc{\Db}{\on{D}^{\on{b}}}
\nc{\Kb}{\on{K}^{\on{b}}}
\nc{\fm}{\mathfrak{m}}
\nc{\Flag}{\on{Flag}}
\nc{\DMmin}{\DM_{\on{min}}}
\nc{\Ddiff}{\on{D}_{\on{diff}}}
\nc{\Dbdiff}{\on{D}^\on{b}_{\on{diff}}}
\nc{\wO}{\widehat{\OO}}
\nc{\wT}{\widehat{T}}
\nc{\from}{\leftarrow}
\nc{\wLL}{\widetilde{\LL}}
\nc{\augCech}{\widetilde{\cC}}
\nc{\Fold}{\on{Fold}}
\nc{\Ext}{\on{Ext}}
\nc{\FF}{\mathbf{F}}
\nc{\Comper}{\Com_{\on{per}}}
\nc{\Unfold}{\on{Unfold}}
\nc{\intHom}{\underline{\Hom}}
\nc{\Ex}{\on{Ex}}
\nc{\tg}{\widetilde{g}}

\def\b{\beta}
\nc{\B}{\mathcal{B}}
\nc{\K}{\mathcal{K}}
\nc{\kos}{\on{Kos}}
\nc{\Perf}{\on{Perf}}
\nc{\tR}{\widetilde{\cR}}
\nc{\X}{\mathcal{X}}
\nc{\Cl}{\on{Cl}}
\nc{\fU}{\mathcal{U}}
\nc{\bU}{\mathbf U}

\nc{\st}{\on{st}}
\def\E{\mathcal{E}}
\nc{\coh}{\on{coh}}

\def\D{\mathcal{D}}

\def\g{\gamma}
\nc{\tU}{\U}
\nc{\bC}{\mathbf{C}}
\nc{\aux}{\on{aux}}

\title{Tate resolutions on toric varieties}
\thanks{The first author was supported by NSF-RTG grant 1502553.  The second author was supported by NSF grants 
DMS-1601619 and DMS-1902123.}
\author{Michael K. Brown}
\address{Department of Mathematics, Auburn University, Auburn, AL}
\email{mkb0096@auburn.edu}
\author{Daniel Erman}
\address{Department of Mathematics, University of Wisconsin, Madison, WI}
\email{derman@math.wisc.edu}
\begin{document}

\maketitle

\begin{abstract}
We develop an analogue of Eisenbud-Fl\o ystad-Schreyer's Tate resolutions for toric varieties. Our construction, which is given by a noncommutative analogue of a Fourier-Mukai transform, works quite generally and provides a new perspective on the relationship between Tate resolutions and Beilinson's resolution of the diagonal. 
We also develop a Beilinson-type resolution of the diagonal for toric varieties.
% and use it to generalize Eisenbud-Fl\o ystad-Schreyer's computationally effective construction of Beilinson monads.
\end{abstract}

%%%%%%%%%%%%%%%%%%%%%%
\section{Introduction}
%%%%%%%%%%%%%%%%%%%%%%
%herzog-iyengar, LP, 
Eisenbud-Fl\o ystad-Schreyer's Tate resolutions are a powerful tool in the study of sheaves on projective space \cite{EFS}, closely connected with both the BGG correspondence~\cite{BGG} and Beilinson monads~\cite{beilinson} and with many applications to commutative algebra, algebraic geometry, and computation~\cite{afo,abo-ranestad, beks-poset, cm, ct, decker-eisenbud, ES,ES-JAMS, es-banks, epsw, smallproj, floystad-notices, fko, schreyer-splitting}.  The main goal of this paper is to develop an analogue of Tate resolutions over projective toric varieties.  
It is reasonable to wonder whether such a theory should even exist, as other related homological features of $\PP^n$--such as the existence of full, strongly exceptional collections of line bundles--do not extend to all projective toric varieties~\cite{HP}.
Indeed, key aspects of Eisenbud-Fl\o ystad-Schreyer's approach to Tate resolutions over $\PP^n$ simply do not generalize to other toric varieties.

We therefore give a totally new construction of Tate resolutions, based on a noncommutative analogue of a Fourier-Mukai transform.
%Our theory vastly generalizes Eisenbud-Erman-Schreyer's theory for products of projective spaces, thereby resolving a natural question arising from~\cite{EFS,EES} and in connection with efforts to extend Boij-S\"oderberg theory to other toric varieties~\cite{EE,ES-CCA}.
Our toric Tate resolutions exhibit both the known features of Tate resolutions on $\PP^n$ and subtle new behavior, such as exactness properties that are parametrized by the combinatorics of the toric variety. The Tate resolutions we obtain over $\PP^n$ recover the original notion of \cite{EFS}, though our
construction is novel even in this case, yielding new proofs of the main results in \cite{EFS} and clarifying the relationship between Tate resolutions and Beilinson's resolution of the diagonal.  We also construct a toric analogue of Beilinson's resolution of the diagonal, generalizing a result of~\cite{CK} and yielding
%generalizes work of Canonaco-Karp over weighted projective stacks \cite{CK} and leads to 
an analogue of Beilinson monads in the toric setting.
%We apply our theory of toric Tate resolutions to give a novel algorithm for computing sheaf cohomology over weighted projective stacks. and to obtain new bounds on multigraded Betti numbers.

%Let us discuss our main results in detail. 

\medskip

Throughout the introduction, $X$ denotes a smooth projective toric variety over a field $k$, and $S = k[x_0, \dots, x_n]$ denotes the Cox ring of $X$, graded by the divisor class group $\Cl(X)$.\footnote{
The smoothness assumption on $X$ can be removed throughout at the cost of working with an associated toric \emph{stack}, rather than a toric variety, and our main results are proven in this level of generality. This is a fairly secondary point, so we stick to the smooth case in the rest of the introduction, unless otherwise noted.}
Tate resolutions are a geometric manifestation of the Koszul duality between a polynomial ring and an exterior algebra, which is often called the \defi{Bernstein-Gel'fand-Gel'fand (BGG) correspondence}. We briefly recall a multigraded version of the BGG correspondence.
% involving the Cox ring $S$. 
Let $E= \Lambda_k(e_0, \dots, e_n)$, equipped with the $\Cl(X) \oplus \Z$-grading where $\deg(e_i) = (-\deg(x_i); -1)$. The Koszul duality between $S$ and $E$ takes the form of an adjunction $\LL\colon \DM(E) \leftrightarrows \Com(S) : \RR$, where $\Com(S)$ is the category of complexes of $\Cl(X)$-graded $S$-modules, and $\DM(E)$ is the category of \defi{differential $E$-modules}, i.e. $E$-modules equipped with a square 0 endomorphism of degree $(0;-1)$ \cite{baranovsky, HHW}. 
See \S \ref{sec:BGGsection} for details.

%To explain our construction of toric Tate resolutions, 

%we begin by recalling the multigraded version of the Bernstein-Gel'fand-Gel'fand (BGG) correspondence, which is an 

%Suppose first that $X = \PP^n$, so that each variable $x_i \in S$ has degree 1. Let $E = \Lambda_k(e_0, \dots, e_n)$ denote the Koszul dual of $S$, with the $\Z$-grading given by $\deg(e_i) = -1$. The BGG correspondence~\cite{BGG} is an adjunction 
%$
%\LL\colon \Com(E) \leftrightarrows \Com(S) \colon \RR
%$
%between the categories of complexes of $E$- and $S$-modules. More generally, if $X$ is a smooth projective toric variety as above,
%\footnote{The multigraded BGG correspondence in great generality.  See Section 2.
%% holds even if $S$ is not the Cox ring of a toric variety
%%Smoothness is not necessary here.
%%in fact, $S$ need not even arise as the Cox ring of a toric variety for The multigraded version of the BGG correspondence to hold. 
%%See Section \ref{sec:BGGsection} for more general statements.
%}, 
%then there is a multigraded version of the BGG correspondence involving the Cox ring $S$ \cite{baranovsky, HHW}. 

%The BGG correspondence is purely algebraic; geometry enters the picture when we pass to Tate resolutions.  
In \cite{EFS}, the \defi{Tate resolution} functor
$
\Tate\colon \Coh(\PP^n) \to \Com(E)
$ 
is introduced as a geometric refinement of the (standard graded) BGG functor $\RR$. 
%Given a coherent sheaf $\F$ on $\PP^n$, $\Tate(\F)$ is an exact, minimal complex of finitely generated free $E$-modules whose Betti numbers encode the dimensions of the cohomology groups $H^i(\PP^n, \cF(j))$ for all $i$ and $j$ \cite[Theorem 4.1]{EFS}.
Tate resolutions provide an explicit link between the BGG correspondence and Beilinson monads,
%'s resolution of the diagonal, 
and they lead to an efficient algorithm for computing sheaf cohomology~\cite{decker-eisenbud}, among many other applications. %Eisenbud-Erman-Schreyer extend the theory of Tate resolutions to products of projective spaces in \cite{EES}; 
 Our main goal is to extend the theory of Tate resolutions to more general toric varieties.
 %see Remark \ref{smoothness}).

Our construction of toric Tate resolutions differs significantly from that of~\cite{EFS}. First, as in the multigraded BGG correspondence, toric Tate resolutions are differential $E$-modules, rather than complexes.
% where $E$ is equipped with the $\Cl(X) \oplus \Z$-grading described above. 
Second, the previous constructions of Tate resolutions simply do not extend to the general toric setting; see Remark \ref{truncations} for just one problem that arises. Instead, we build toric Tate resolutions via a Fourier-Mukai transform associated to the diagram
$$
\xymatrix{
&X \times ``\Spec(E)"  \ar[ld]_-{\pi_1}\ar[rd]^-{\pi_2}&\\
X&&``\Spec(E)".
}
$$
Of course, this diagram doesn't fully make sense: $E$ is a noncommutative ring, and so ``$\Spec(E)$'' is not well-defined. 
But, one may define the toric Tate resolution functor $\Tate$ via the corresponding categorical diagram
%may define a Fourier-Mukai transform given by the categorical diagram
\begin{equation}
\label{eqn:tate def}
\xymatrix{
&\Coh(X_E)\ar[rr]^-{-\otimes \mathcal K}&&\DM(X_E)\ar[rd]^{R\pi_{2*}}&\\
\Coh(X)\ar[rrrr]^-{\mathbf T}\ar[ru]^{\pi_1^*}&&&&K_{\DM}(E).
}
\end{equation}
Here, $\Coh(X_E)$ (resp. $\DM(X_E)$) is the category of coherent (resp. differential) $\cO_X$-$E$ bimodules, $K_{\DM}(E)$ is the homotopy category of differential $E$-modules, and the object $\K$ is a Koszul complex;
%Letting $K$ denote the Koszul complex on the variables of $S$, the object $\K \in \DM(X_E)$ is the sheaf $\K = \bigoplus_{d\in \Cl(X)} \widetilde{K}(d)$, 
see \S \ref{mainresults} for a full explanation of the categories and functors in \eqref{eqn:tate def}.

%In more detail: the Koszul complex $K$ on the variables of $S$ may be identified with $S \otimes_k E$; it follows that the sheaf $\K = \bigoplus_{d\in \Cl(X)} \widetilde{K}(d)$, equipped with the is a differential $\cO_X$-$E$ bimodule. 

%Let 

%That is, given a coherent sheaf $\cF$ on $X$, we define the Tate resolution of $\F$ to be
%\begin{equation}
%\Tate (\cF) := R\pi_{2*}( \pi_1^* \cF \otimes \cK).
%\end{equation}
%Unlike the original construction of Tate resolutions on $\PP^n$, 
Our main result 
%on toric Tate resolutions 
shows that toric Tate resolutions have properties analogous to those 
%of Tate resolutions 
over $\PP^n$
(see Theorem \ref{thm:toric exact} for a statement including the nonsmooth case):

\begin{thm}
\label{thm:toric exact intro}
Let $X$ be a smooth, projective toric variety with Cox ring $S$, and let $E$ be the $\Cl(X) \oplus \Z$-graded Koszul dual exterior algebra of $S$, as in the multigraded BGG correspondence. Let $\F$ be a coherent sheaf on $X$. The Tate resolution $\Tate (\F)$, defined as in ~\eqref{eqn:tate def}, has the following properties:
\begin{enumerate}
	\item  $\Tate(\F)$ is a minimal, exact, free differential $E$-module.
	\item  The Tate resolution encodes the sheaf cohomology groups of $\F$: more precisely, for any $a \in \Cl(X)$ and $j\in \ZZ$, we have
	$
	H^j(X , \F(a)) = \underline{\Hom}_E(k, \Tate (\F))_{(a;-j)}.
	$
	\item Let $B$ denote the irrelevant ideal of $X$, and let $M$ be a $\Cl(X)$-graded $S$-module such that $\widetilde{M} = \F$. Assume $H^0_B(M) = 0$. The injective map $M \into \bigoplus_{a \in \Cl(X)} H^0(X, \F(a))$ induces an embedding $\RR(M) \into \Tate(\F)$ of differential $E$-modules.
\end{enumerate}
\end{thm}

Theorem \ref{thm:toric exact intro} shows that our toric Tate resolutions have the same key characteristics as those of~\cite{EFS}.
There is a categorical shift to differential modules, but when $X=\PP^n$, the ``folding" and ``unfolding" functors from \S \ref{sec:folding} recover the original results.  For toric varieties other than $\PP^n$, Tate resolutions also have rich, new structures; the reader may want to look to \S\ref{subsec:tate examples}, where several examples of toric Tate resolutions are worked out in detail.  

Moreover, our Fourier-Mukai construction provides an independent--and conceptually quite different--proof of the main results on Tate resolutions from~\cite{EFS}.  For instance, Eisenbud-Fl{\o}ystad-Schreyer's proof of Theorem \ref{thm:toric exact intro}(2) over $\PP^n$ relies upon their theory of the linear part of a free complex, as developed in~\cite[\S3]{EFS}; by contrast, Theorem \ref{thm:toric exact intro}(2) is an elementary consequence of our construction.  Our work also reveals a deeper connection between Tate resolutions and Beilinson's resolution of the diagonal: when $X = \PP^n$, the Fourier-Mukai kernel $\K$ ``lifts'' Beilinson's resolution of the diagonal from $\PP^n \times \PP^n$ to ``$\PP^n \times E"$. See \S\ref{othertorics} for details, and see \S\ref{BMsection} for how this generalizes to other toric varieties.

%  explaining this precisely would require some additional notation and results, so we refer the reader to Remark~\ref{othertorics}, and especially Diagram \eqref{keydiagram}.\daniel{I still want to wrodsmith this last sentence.}
%The $\Z$-grading on 
%%the exterior algebra 
%$E$ 
%%plays a key role in Theorem \ref{thm:toric exact intro}, 
%links generators of $\Tate(\cF)$ to specific cohomology groups; that is, the $\Z$-grading on $E$ plays role that the homological grading does in~\cite[Theorem~4.1]{EFS}.
% and a homological perturbation argument using \cite[Lemma 3.5]{EFS}. 
 %
 
%While this is known to hold over $\PP^n$, there is a subtle implication in the general toric case, as there can be distinct toric varieties with the same effective cone; see Remark~\ref{rmk:intro Mori}. \michael{I think I'd like to just remove this last paragraph. We make this point in Remark 1.4 anyway.}  \daniel{let's combine them.  i'll start on it} 
%This has a subtle implication.  A theme in research on syzygies is that the geometry of $\PP^n$ can influence the algebraic properties of graded free resolutions.
 
\medskip 

The toric variety $X$ is determined by its Cox ring $S$ and its irrelevant ideal. Since the Tate resolution is a fundamentally geometric object, it is natural to ask:
%When $\Cl(X)\ne \ZZ$, there can be non-homotopic ways to extend $\RR(M)$ to a minimal, exact, free differential module.  As $\RR(M)$ is purely algebraic, whereas $\Tate(\widetilde{M})$ involves the geometry of $X$, this  is analogous to the fact that one can obtain different toric varieties with the same Cox ring by altering the irrelevant ideal. This leads to the question: 
how is the irrelevant ideal of $X$ reflected in the properties of toric Tate resolutions? Our next result provides the answer, demonstrating a phenomenon that was not present on $\PP^n$.  To state the result, we need the following definition: if $I\subseteq \{0,1,\dots,n\}$, we set $P_{I} := \langle x_i \text{ : } i\in I\rangle\subseteq S$, and we say that $I$ is {\bf irrelevant} if $P_I$ contains the irrelevant ideal of $X$.

% is totally is that Tate resolutions satisfy an array of exactness properties encoded by the irrelevant ideal. In detail: 
% given $I\subseteq \{0,1,\dots,n\}$, we set $P_{I} := \langle x_i \text{ : } i\in I\rangle\subseteq S$, and we say that $I$ is {\bf irrelevant} if $P_I$ contains the irrelevant ideal.

%Saying ``the" Tate resolution is misleading, as our construction of Tate resolutions is, \emph{a priori}, only well-defined up to homotopy equivalence; see \S \ref{sec:stacks} for details. 
%\daniel{Why emphasize this in the intro?  Why not just use ``a Tate resolution'' etc throughout intro and make this point at the end?}
%The issue is that, in contrast to what happens over $\PP^n$, the Tate resolution of $\widetilde{M}$ is not entirely determined by $\RR(M)$. 
%This is because, when $\Cl(X)\ne \ZZ$, there can be non-isomorphic {\color{blue} homotopic?} extensions of $\RR(M)$ that have the properties in Theorem~\ref{thm:toric exact intro}(1). As $\RR(M)$ is purely algebraic, whereas $\Tate(\widetilde{M})$ involves the geometry of $X$, this difficulty is analogous to the fact that one can obtain different toric varieties with the same Cox ring by altering the irrelevant ideal.

\begin{thm}[Exactness properties]\label{thm:exactness properties}
Let $I\subseteq \{0,1,\dots,n\}$, and let $E_I = E / \langle e_i \text{ : } i \notin I \rangle$.   If $I$ is an irrelevant subset, then, for any coherent sheaf $\F$ on $X$, $\Tate(\cF)\otimes_E E_I$ is exact.
\end{thm}

Thus, not only are Tate resolutions exact, but this exactness is robust and nuanced.  For instance, if $X$ is the blowup of $\mathbb P^1\times \mathbb P^1$ at a torus fixed point, then the irrelevant ideal is
\[
\langle x_0,x_2\rangle \cap \langle x_1,x_3\rangle \cap \langle x_2,x_4\rangle \cap \langle x_0,x_3\rangle \cap \langle x_1,x_4\rangle.
\]
The theorem says that any Tate resolution $\Tate(\cF)$ for a sheaf on $X$ will remain exact even if we restrict the differential to just the variables $\{e_0,e_2\}$, or to just $\{e_1,e_3\}$, and so on.
% is a projective bundle over some $\PP^n$, then Tate resolutions on $X$ will remain exact upon restricting the differential to only the variables in $E$ dual to the coordinates on either the base or the fiber. 
This result highlights the fact that Tate resolutions depend not just on the Cox ring of $X$, but also on the combinatorics of the irrelevant ideal; the examples in \S\ref{subsec:tate examples} illustrate this point in some detail. These exactness properties are invisible over $\PP^n$, as there are no nontrivial irrelevant subsets of $\{0,1, \dots, n\}$ in that case. However, a precursor to this result can be found in the ``exact strands'' result from~\cite[\S3]{EES} for products of projective spaces. 
%We expect that, in general, Properties (1) - (3), along with the exactness properties in Theorem \ref{thm:exactness properties}, determine the Tate resolution up to isomorphism of differential $E$-modules: see Question \ref{q:intrinsic}. 

%\michael{I'm now thinking we should put this remark somewhere in the body. It's too inside baseball for the intro. I think the value it adds is outweighed by making this already intimidating paper look more impenetrable. However, I think we should add a reference to this remark to the intro, after the exactness properties. Like: ``see Remark ??? for additional ways in which the combinatorics of the irrelevant ideal are manifested in the Tate resolution"}
%
%\michael{Huge change here. I put most of the resolution of the diagonal stuff into the body and just mention it here. We can explain in more detail here why it's important, and add some references. But I think the construction of the resolution is too technical to be in the intro; I think it was bogging down the narrative.}
%

%%%%%%%%%%%%%%%%%%%%%%%%%%%%%
%\subsection*{A resolution of the diagonal for toric varieties}
%%%%%%%%%%%%%%%%%%%%%%%%%%%%%

\medskip

A main result of \cite{EFS} is that a coherent sheaf on $\PP^n$ can be completely recovered from its Tate resolution \cite[Theorem 6.1]{EFS}. We generalize this to toric varieties in Corollary~\ref{monad}. En route to this result, we develop a toric analogue of Beilinson's resolution of the diagonal in Theorem \ref{resdiag}. Our resolution of the diagonal is based on a fairly simple idea.  Let $S' = k[x_0, \dots, x_n, y_0, \dots, y_n]$ be the Cox ring of $X\times X$. The image of the diagonal embedding $X \into X \times X$ is essentially the locus where ``$x_i=y_i$'' for all $i$, though those equations are not homogeneous and hence do not yield well-defined equations on $X\times X$. However, we can force these equationsfan
to be homogeneous by imposing them not on $S'$ itself, but on a certain $S'$-module.  This recovers a result of \cite{CK} and has implications for the study of derived categories of toric varieties and virtual resolutions in multigraded commutative algebra; see also Conjecture~\ref{c:finite res} and the surrounding discussion.

\subsection*{Overview}
%%%%%%%%%%%%%%%%%%%%
In \S \ref{sec:BGGsection}, we discuss the multigraded BGG correspondence in detail. While many of the results in this section are known, some key results, like Proposition \ref{tor},
are new. In \S \ref{sec:stacks}, we introduce our construction of the toric Tate resolution, and we prove Theorems~\ref{thm:toric exact intro} and \ref{thm:exactness properties}. In \S \ref{sec:res diagonal}, we construct our toric resolution of the diagonal, and we explain in \S \ref{BMsection} how one can recover a sheaf from its toric Tate resolution. As an application of our results, we give in \S \ref{buchweighted} an interpretation of the bounded derived category of a weighted projective stack in terms of its Koszul dual exterior algebra. Finally, we discuss some questions and conjectures in \S \ref{sec:applications}. We provide in Appendix \ref{sec:positivegrading} some background on the notion of a ``positively" multigraded ring, and Appendix \ref{DM} is a collection of necessary technical results on differential modules.

%%%%%%%%%%%%%%%%%%%% 
\subsection*{Acknowledgements} 
Many ideas in this paper had origins in joint conversations with David Eisenbud and Frank-Olaf Schreyer, and we are tremendously grateful for their generous insights.  This project evolved over several years, and along the way we benefited from conversations with many researchers including: Christine Berkesch, Juliette Bruce, Lauren Cranton Heller, Milena Hering, Lutz Hille, Srikanth Iyengar, Brian Lehmann, Diane Maclagan, Steven V Sam, Mahrud Sayrafi, Gregory G. Smith,  Hal Schenck, Mike Stillman, and Mark Walker. The computer algebra system {\verb Macaulay2 }~\cite{M2} provided valuable assistance. Finally, we thank the anonymous referees for their careful reading and their many suggestions that have improved this paper.
%%%%%%%%%%%%%%%%%%%%

%%%%%%%%%%%
\section{The multigraded BGG correspondence}
\label{sec:BGGsection}
%%%%%%%%%%%
%%%%%%%%%%%%%%%%%
We recall the multigraded BGG correspondence discussed in the introduction, and we prove a number of related results. 
Throughout this section, $S$ denotes a polynomial ring $k[x_0, \dots, x_n]$ that is positively graded by some abelian group $A$ (Definition \ref{posring}). In particular, $S$ may be the Cox ring of a projective toric variety (Example \ref{cox}). As in the introduction, we let $E = \Lambda_k(e_0, \dots, e_n)$, equipped with the $A \oplus \Z$-grading given by $\deg(e_i) = (-\deg(x_i); -1)$. We call the $\Z$-grading on $E$ the \defi{auxiliary grading}. 

\begin{conventions}
\label{leftconvention}
Throughout the paper, all $E$-modules are right modules.  However, since $E$ is graded commutative (with respect to the auxiliary grading), any right $E$-module $M$ can be considered as a left $E$-module. In detail: the left $E$-action on a right $E$-module $M$ is given by $em = (-1)^{\on{aux}(e)\on{aux}(m)}me$, where $\on{aux}( - )$ denotes the degree in the auxiliary grading. Given a graded ring $R$, we let $\underline{\Hom}_R( - , -)$ denote the internal $\Hom$ in the category of graded $R$-modules.
\end{conventions}
%%%%%%%%%%%%
\subsection{Differential $E$-modules}
\label{sec:folding}
%%%%%%%%%%%%

\begin{defn}
A \defi{differential $E$-module} is an $A \oplus \Z$-graded right $E$-module $D$ equipped with a degree $(0; -1)$ endomorphism $\del$ such that $\del^2 = 0$. We restrict attention to differentials with this particular degree because they are the only ones that will arise in our work.  See~\cite{BE} for a more general treatment of differential modules.\footnote{Differential modules also appear in Rouquier's analogue of the BGG correspondence for non-graded polynomial and exterior algebras in \cite[Section 4]{rouquier}.} We have $\del(d e) = \del(d)e$ for all $d \in D$ and $e \in E$; bearing in mind Conventions \ref{leftconvention}, we also have $\del(ed) = (-1)^{\aux(e)}e\del(d)$.

A \defi{morphism} $D \to D'$ of differential $E$-modules is a degree 0 map $f \colon D \to D'$ satisfying $f \del = \del'  f$. Let $\DM(E)$ denote the category of differential $E$-modules. The \defi{homology} of an object $D \in \DM(E)$ is the subquotient $\ker(\del \colon D \to D(0; -1) / \im(\del \colon D(0;1) \to D)$, denoted $H(D)$. A morphism in $\DM(E)$ is a \defi{quasi-isomorphism} if it induces an isomorphism on homology. A \defi{homotopy} of morphisms $f, f' \colon D \to D'$ in $\DM(E)$ is a morphism $h \colon D\to D'(0;1)$ of $E$-modules such that $f - f' = h \del + \del' h$. The \defi{mapping cone} of a morphism $f \colon D \to D'$ in $\DM(E)$ is the module $D' \oplus D(0,-1)$ equipped with the differential
$
\begin{pmatrix} \del' & f \\
0 & -\del \end{pmatrix}.
 $
 
%\[
%\bordermatrix{
%&D'&D(0;-1)\cr
%D'(0;-1)& \del' & f \cr
% D(0;-2)& 0 & -\del 
% }.
% \] 
\end{defn}

\begin{rem}
\label{dgmodules}
The ring $E$ may be considered as an $A$-graded dg $k$-algebra with trivial differential and homological grading induced by the auxiliary $\Z$-grading. The category $\DM(E)$ is equivalent (in fact, isomorphic) to the category of dg-modules over this dg-algebra. This is the perspective taken in the proof of the multigraded BGG correspondence in \cite{HHW}. We use differential modules because they are more amenable to computing in {\verb Macaulay2 }.
% {and because of the connection with~\cite{ABI,BE}.} 

Yet another way of thinking of the category $\DM(E)$ is as follows. Let $\Comper(E)$ denote the category of complexes of $E$-modules of the form
$
\cdots \to N(0; 1) \xra{\del} N \xra{\del} N(0; -1) \to \cdots,
$
where the differentials are all identical.  A morphism 
%in this category 
is a chain map that is the same in each homological degree. There is an equivalence (in fact, isomorphism) of categories
$
\Ex : \DM(E) \xra{\simeq} \Comper(E)
$
that sends a differential module $(D, \partial)$ to the complex
\[
\Ex(D) = (\cdots \to D(0;1) \xra{\partial} D \xra{\partial}  D(0;-1) \to \cdots).
\]
We call $\Ex(D)$ the \defi{expansion} of $D$, following the terminology in \cite{ABI}. 
\end{rem}
\begin{defn}
\label{deriveddm}
The \defi{derived category of differential $E$-modules} $\on{D}_{\DM}(E)$ is obtained by inverting quasi-isomorphisms in $\DM(E)$. The category $\on{D}_{\DM}(E)$ is triangulated, with shift functor given by $D[1] = D(0; 1)$; to prove this, one may use the identification between differential $E$-modules and dg-modules over a certain dg-algebra explained in Remark \ref{dgmodules}. 
The \defi{bounded derived category of differential $E$-modules} $\Db_{\DM}(E)$ is the subcategory of $\on{D}_{\DM}(E)$ given by objects with finitely generated homology. Denote by $\on{D}^{\on{fg}}_{\DM}(E)$ the subcategory of $\Db_{\DM}(E)$ given by objects whose underlying module is finitely generated.
\end{defn}

\begin{prop}
\label{fg}
The inclusion $\on{D}^{\on{fg}}_{\DM}(E) \into \Db_{\DM}(E)$ is an equivalence.
\end{prop}

\begin{proof}
Let $D \in  \Db_{\DM}(E)$. Since $S$ is positively $A$-graded, $E$ is as well; choose a positive $A$-grading on $E$, as in Appendix \ref{sec:positivegrading}. For the rest of this proof, we will consider $E$ as a nonnegatively $\Z$-graded $k$-algebra and $D$ as a $\Z$-graded $E$-module; notice that the differential on $D$ is degree 0 with respect to this grading. By \cite[Theorem 1.3(b)]{BE}, $D$ admits a free resolution $F$ such that $\dim_k F_j < \infty$ for all $j$, and $F_j = 0$ for $j \ll 0$. Choose $N \gg 0$ such that the homology of $F$ lives in degrees $< N$. Let $F'$ be the $E$-submodule of $F$ generated by the elements of degree $< N$. The submodule $F'$ is finitely generated, and it is in fact a differential submodule of $F$, since the differential on $F$ is degree 0. Finally, observe that the inclusion $F' \into F$ is a quasi-isomorphism.
\end{proof}

\begin{defn}[Folding and unfolding differential modules]
\label{defn:fold}
Suppose $S$ is equipped with the standard grading, i.e. $\deg(x_i) = 1$ for all $i$. Denote by $\Com(E)$ the category of complexes of $\Z$-graded $E$-modules.
% note that $\Com(E)$ does not see the auxiliary $\Z$-grading on $E$.
 The \defi{folding functor}
$
\on{Fold}
\colon \Com(E) \to \DM(E)
$
is given by $(\cdots \to C_j \xra{\del_C} C_{j-1} \to \cdots) \mapsto (\bigoplus_{j \in \Z} C_j(0; -j), \del_C).$ Here, we think of each $\Z$-graded module $C_j$ as $\Z \times \Z$-graded by setting $(C_j)_{(a,b)} = \begin{cases} (C_j)_a  & a = b, \\ 0 & a \ne b. \end{cases}$

\noindent Going the other direction, if $D \in \DM(E)$, set
$
D_j = \{d \in D \text{ : } \deg(d) = (a; i), \text{ where } i - a = j\}.
$
 Notice that the $D_j$ are $E$-modules. Since $\del_D$ is a map from $D$ to $D(0; -1)$, $\del_D$ induces a map from $D_j$ to $D_{j -1}$ for all $j$. Note that any $\Z \times \Z$-graded $E$-module $M$ can be considered as a $\Z$-graded $E$-module with components $M_a = \bigoplus_{i  \in \Z} M_{(a; i)}$. The \defi{unfolding functor} 
$
\on{Unfold}\colon \DM(E) \to \Com(E)
$
is given by $D \mapsto (\cdots \xra{\del_D} D_j \xra{\del_D} D_{j-1} \xra{\del_D} \cdots ).$
\end{defn}

One easily verifies that the functors $\Fold$ and $\on{Unfold}$ are inverse equivalences.

%See Appendix \ref{DM}, or \cite{ABI,BE}, for additional background on differential $E$-modules.
%%%%%%%%%%%%%%%%%%%%%%
\subsection{The multigraded BGG functors}
%%%%%%%%%%%%%%%%%%%%%%%
Let $\Com(S)$ denote the category of complexes of $A$-graded $S$-modules. We begin by defining the BGG functors
$
\LL : \DM(E) \leftrightarrows \Com(S) : \RR.
$
Given $D \in \DM(E)$, the complex $\LL(D)$ has terms and differential given by:
$$
\LL(D)_j = \bigoplus_{a \in A}S(-a)  \otimes_k D_{(a; j)} \quad \text{ and } \quad s \otimes d \mapsto  (\sum_{i = 0}^n x_is \otimes e_id) - s \otimes \del_D(d).
$$
Here, while $D$ is a right $E$-module, the $e_i$ act on $d$ on the left via the formula in Conventions~\ref{leftconvention}.  Let $\om_E$ denote the $E$-module $\underline{\Hom}_k(E, k)\cong E(-\sum_{i = 0}^n \deg(x_i); -n-1)$. Given $C \in \Com(S)$, the object $\RR(C) \in \DM(E)$ has underlying $E$-module and differential
$$
\bigoplus_{j \in \Z} \bigoplus_{a \in A} (C_j)_a \otimes_k \om_E(-a; -j)  \quad \text{ and } \quad c \otimes f \mapsto (-1)^j (\sum_{i = 0}^n x_ic\otimes e_if) + \del_C(c)\otimes f,
$$ 
where we have assumed that $c$ lies in $C_j$.
%and the differential acts on the $j^{\th}$ summand $\bigoplus_{a \in A} (C_j)_a \otimes_k \om_E(-a; -j)
%$ by 
%$$
%c \otimes f \mapsto (-1)^j (\sum_{i = 0}^n x_ic\otimes e_if) + \del_C(d) \otimes f .
%$$
The following Theorem follows from work of Hawwa-Hoffman-Wang \cite{HHW}:

\begin{thm}
\label{BGGthm}
The functors $\LL$ and $\RR$ satisfy the following.
\label{adjoint}
\begin{enumerate}
\item[(a)] The functor $\LL$ is left adjoint to $\RR$, and both $\LL$ and $\RR$ are exact.
\item[(b)] Let $C \in \Com(S)$ and $D \in \DM(E)$. The unit and counit of adjunction give quasi-isomorphisms $D \xra{\simeq} \RR \LL(D)$ and $\LL \RR(C) \xra{\simeq} C$.
\item[(c)] The functors $\LL$ and $\RR$ induce an equivalence $\on{D}_{\DM}(E) \simeq \on{D}(S)$.
\end{enumerate}
\end{thm}

\begin{example}
\label{weightedexample}
Given pairwise relatively prime positive integers $w_0, \dots, w_n$, the \defi{weighted projective stack} $\PP(w_0, \dots, w_n)$ is the stack quotient of $\A^{n+1} \setminus \{0\}$ by the $\G_m$-action $t \cdot (a_0, \dots, a_n) = (t^{w_0}a_0, \dots, t^{w_n}a_n)$ (cf. \cite[Definition 2.1.1]{AH}). Say $S = k[x_0, x_1, x_2]$ is equipped with the $\Z$-grading given by $\deg(x_0) = 1 = \deg(x_1)$ and $\deg(x_2) = 2$. Geometrically, $S$ is the Cox ring of the weighted projective stack $\PP(1,1,2)$. The differential module $\RR(S)$ is infinitely generated, and the degrees of its generators correspond to the degrees $i \in \Z$ such that $S_i \ne 0$. It has the form:
\vskip\baselineskip
$$
\xymatrix{
\om_E \ar[r]\ar@/_-1.5pc/[rr]& \om_E(-1;0)^{\oplus 2}\ar[r]\ar@/_1.5pc/[rr] & \om_E(-2;0)^{\oplus 4}\ar[r]&\dots.
}
$$
\vskip\baselineskip
\noindent 
As bases, we choose the monomials $\{1\}, \{x_0,x_1\}$ and $\{x_0^2,x_0x_1,x_1^2, x_2\}$.  The horizontal arrows are multiplication by $e_0$ and $e_1$, and the curved arrows are multiplication by $e_2$;
%For instance, if we write down the differential restricted to these summands, then it has the form
%\footnotesize
%\[
%\bordermatrix{
%&1&&x_0&x_1&&x_0^2&x_0x_1&x^2&x_2\cr
%1&&&&&&&&&\cr
%x_0&e_0&&&&&&&&\cr
%x_1&e_1&&&&&&&&\cr
%x_0^2&&&e_0&&&&&&\cr
%x_0x_1&&&e_1&e_0&&&&&\cr
%x_1^2&&&&e_1&&&&&\cr
%x_2&e_2&&&&&&&&\cr
% }.
%\]
%\normalsize
for instance, $\omega_E$ maps to $\om_E(-1;0)^{\oplus 2}$ via $(e_0,e_1)^t$ and to $\om_E(-2;0)^{\oplus 4}$ via $(0,0,0,e_2)^t$.
\end{example}

\begin{remarks} We highlight the following observations:
\label{RRbicomplex}
\text{ }
\begin{enumerate}
\item Given $C \in \Com(S)$, it can be helpful to interpret $\RR(C)$ as the totalization of a bicomplex. Form a bicomplex with $q^{\th}$ row given by 
$$
\cdots \from \RR(C_q)(0; -1) \xla{(-1)^q \sum_{i = 0}^n x_i \otimes e_i} \RR(C_q) \xla{(-1)^q \sum_{i = 0}^n x_i \otimes e_i} \RR(C_q)(0; 1) \from \cdots
$$
and vertical differentials induced by $\del_C$. Notice that the $q^{\th}$ row is the expansion of $\RR(C_q)$, up to a sign on the differential (see Remark \ref{dgmodules}). Totalizing this bicomplex gives an object in $\Comper(E)$; applying the inverse of the expansion functor gives the differential module $\RR(C)$. Notice that the homological grading on $C$ plays a crucial role here; for this reason, the functor $\LL$ is not given by the totalization of a bicomplex.
\item When $\deg(x_i) = 1$ for all $i$, the multigraded BGG functors are essentially the same as the original ones from \cite{EFS}. More precisely: letting $\LL_{\on{classical}}$ and $\RR_{\on{classical}}$ denote the original BGG functors over $\PP^n$, and using the notation of Definition~\ref{defn:fold}, we have
$
\RR = \on{Fold} \circ \RR_{\on{classical}}$ and $\LL = \LL_{\on{classical}} \circ \on{Unfold}.
$
\item Given a graded $S$-module $M$, the object $\RR(M)$ in $\DM(E)$ is an injective coflag, which is like an injective resolution in the setting of differential modules; see Definition~\ref{freeres}. 
\item One could define similar functors between $\Com(E)$ and $\DM(S)$ or 
between 
$\DM(E)$ and $\DM(S)$. However, our primary interest is in complexes of $S$-modules, and so it is natural (for us) to preserve the homological grading on the $S$-module side.  
\end{enumerate}
\end{remarks}

\begin{example}
\label{ex:RR hirz}
Let $X=\mathbb F_3$ be the Hirzebruch surface of type $3$. Let $S=k[x_0,x_1,x_2,x_3]$ be its Cox ring, where the $\Cl(X)=\ZZ^{\oplus 2}$-grading is given by
$\deg(x_0)=(1,0), \deg(x_1)=(-3,1), \deg(x_2)=(1,0),$ and 
$\deg(x_3)=(0,1).$ The degrees of the generators of $\RR(S)$ correspond to the degrees of the effective cone of $X$, as illustrated by the shaded region below:
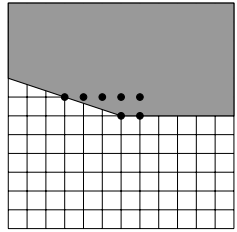
\[
\begin{tikzpicture}[scale = .25]
\draw[step=1cm,black,very thin] (0,0) grid (12,12);
\draw[fill=gray!80] (6,6)--(0,8)--(0,12)--(12,12)--(12,6)--(6,6);
\filldraw[black] (6,6) circle (5pt);
\filldraw[black] (6,7) circle (5pt);
\filldraw[black] (5,7) circle (5pt);
\filldraw[black] (4,7) circle (5pt);
\filldraw[black] (3,7) circle (5pt);
%\filldraw[black] (3,7) circle (5pt);
\filldraw[black] (7,7) circle (5pt);
\filldraw[black] (7,6) circle (5pt);
\end{tikzpicture}
\]
Focusing on the free summands of $\RR(S)$ corresponding to the degrees marked with black dots,  we get the diagram in \eqref{eqn:RR hirz} below.  The twists in the auxiliary grading are all 0, so we have omitted them.
 The arrows indicate the effect of the differential:
% \small{

\begin{equation}\label{eqn:RR hirz}
\scriptsize
\xymatrix{
\ddots&\vdots&\vdots&\vdots&\vdots&\\
\omega_E(3,-1) \ar[r]\ar[u]&\omega_E(2,-1)^{\oplus 2}\ar[u]\ar[r]&\omega_E(1,-1)^{\oplus 3}\ar[u]\ar[r]&\omega_E(0,-1)^{\oplus 5}\ar[u]\ar[r]\ar[lllu]&\omega_E(-1,-1)^{\oplus 7} \ar[u]\ar[r]\ar[lllu]&\cdots\ar[lllu]\\
0&0&0&\omega_E \ar[u]\ar[r]\ar[lllu]&\omega_E(-1,0)^{\oplus 2} \ar[u]\ar[r]\ar[lllu]&\cdots\ar[lllu]\\
}
\end{equation}

%}
\normalsize
\noindent For instance, if we choose the monomials in $S$ as generators for the free module $\RR( S)$, then the differential sends $\omega_E$ to:  $\omega_E(3,-1)$ via $e_1$, to a summand of $\omega_E(0,-1)^{\oplus 5}$ via $e_3$; and to $\omega_E(-1,0)^{\oplus 2}$ via the column vector $(e_0,e_2)^t$. 
\end{example}

The following calculation of the homology of the functors $\LL$ and $\RR$ extends \cite[Proposition 2.3]{EFS}. See Appendix \ref{tensorhom} for the definitions of tensor product, internal $\Hom$, and $\Ext$ for differential modules used in the statement of this result and its proof.

\begin{prop}
\label{tor} 
Let $C \in \Com(S)$ and $D \in \DM(E)$. Assume $D$ is finitely generated as an $E$-module. We have
\begin{enumerate}
\item[(a)] $H(\RR(C))_{(a; j)} = \on{Tor}_j^S(C, k)_a$, and
\item[(b)] $H_j(\LL(D))_a=\Ext^{\DM}_E(k, D)_{(a; j)}$.
\end{enumerate}
\end{prop}

\begin{proof}
We can view the Koszul complex $K$ on the variables of $S$ as the complex of $A$-graded $S$-modules with $j^{\th}$ term $\bigoplus_{d \in A} S(-d) \otimes_k (\om_E)_{(d; j)}$ and differential given by multiplication on the left by $\sum_{i = 0}^n x_i \otimes e_i$. We have
\begin{align*}
\RR(C)_{(a; j)} &= \bigoplus_{i \in \Z} \bigoplus_{d \in A} (C_i)_{d} \otimes_k (\om_E)_{(-d+a; j - i)}  \cong (\bigoplus_{ i \in \Z} C_i \otimes_S K_{j - i})_a  = ((C \otimes_S K)_{j})_a.
\end{align*}
Thus, $\RR(C) \cong C \otimes_S K$ as $A \oplus \Z$-graded $k$-vector spaces. Moreover, this isomorphism preserves the differentials; this proves (a). As for (b), setting $w = \sum_{i = 0}^n \deg(x_i)$, we have
\begin{align*}
(\LL(D)_j)_a &= \bigoplus_{d \in A} S_{-d} \otimes_k D_{(d + a; j)}\\
& \xra{\sigma}  \bigoplus_{d \in A} S_{-d} \otimes_k (\om_E(d + w; n+1) \otimes_E D)_{(a;j)} = (\RR(S(w)[n+1]) \otimes^{\DM}_E D)_{(a; j)},
\end{align*}
where the $k$-linear isomorphism $\sigma$ sends $s \otimes d$ to $(-1)^{n |s| + j} s \otimes (x_0 \cdots x_n) \otimes d$. Here, $|s|$ denotes the degree of $s$ in the standard grading on $S$, and $x_0 \cdots x_n$ is the element of $\om_E$ dual to $e_0 \cdots e_n \in E$. This gives an isomorphism $\LL(D) \cong \RR(S(w)[n+1]) \otimes_E^{\DM} D$ of $A \oplus \Z$-graded $k$-vector spaces that preserves the differentials.

Given $N \in \DM(E)$, let $N^\vee = \underline{\Hom}^{\DM}_E(N, E)$. Since $D$ is finitely generated, we have an isomorphism $(\RR(S(w)[n+1]) \otimes^{\DM}_E D)_{(a; j)} \cong \underline{\Hom}_E^{\DM}(\RR(S(-w)[n+1])^\vee, D)_{(a;j)}$ in $\DM(E)$. Note that $\RR(S(-w)[n+1])^\vee$ is a free resolution (in the sense of Definition \ref{freeres}) of the residue field $k$, considered as an object in $\DM(E)$ with trivial differential. This proves (b).
\end{proof}

\begin{remark}
In the above proof, the finitely generated assumption is used when pulling a direct sum out of the first component of $\Hom$.  This subtlety does not arise in \cite{EFS}, due to the presence of the homological grading.
\end{remark}

\begin{cor}\label{cor:RM high degree} Let $C \in \Com(S)$. If $C$ is bounded, and each term of $C$ is finitely generated, then $\dim_kH(\RR(C)) < \infty$.
\end{cor}

\begin{cor}
\label{cor:bounded}
The equivalence $\on{D}_{\DM}(E) \simeq \on{D}(S)$ induced by $\LL$ and $\RR$ in Theorem \ref{BGGthm}(c) restricts to an equivalence $\on{D}^{\on{b}}_{\DM}(E) \simeq  \Db(S)$. 
\end{cor}

\begin{proof}
Immediate from Proposition \ref{fg} and Corollary \ref{cor:RM high degree}.
\end{proof}

We next observe that a multigraded generalization of Eisenbud-Fl\o ystad-Schreyer's Reciprocity Theorem \cite[Theorem 3.7(a)]{EFS} follows immediately from Theorem \ref{BGGthm}:

\begin{thm}
\label{reciprocitybody}
Let $S$ be a positively graded polynomial ring over $k$, and let $E$ denote its Koszul dual exterior algebra. Let $M$ be a finitely generated $S$-module and $N$ a finitely generated $E$-module. The complex $\LL(N)$ is a free resolution of $M$  if and only if the differential module $\RR(M)$ is an injective resolution of $N$ (see Definition \ref{freeres} for the notion of an injective resolution of a differential module). 
\end{thm}

\begin{proof}
If there is a quasi-isomorphism $\LL(N) \xra{\simeq} M$, then, since $\RR$ is exact (Theorem \ref{BGGthm}(a)), we have a quasi-isomorphism $\RR \LL(N) \xra{\simeq} \RR(M)$. Composing with the quasi-isomorphism $N \xra{\simeq} \RR \LL(N)$ from Theorem \ref{BGGthm}(b) gives a quasi-isomorphism $N \xra{\simeq} \RR(M)$. Since $\RR(M)$ is an injective coflag (Remarks \ref{RRbicomplex}(3)), $N \xra{\simeq} \RR(M)$ is an injective resolution. The converse is similar.
\end{proof}

\begin{remark}
\label{rmk:hirz non linear and qiso}
A main emphasis of \cite{EFS} is the relationship between BGG and linear complexes over polynomial and exterior algebras~\cite[Section 3]{EFS}. While some of this theory extends naturally to the multigraded setting, 
 there are also significant distinctions; we explore this in detail in our follow-up papers \cite{linear, positivity}. For instance, one of the fundamental facts used by Eisenbud-Fl{\o}ystad-Schreyer in their construction of Tate resolutions for $\PP^n$ is that, for any finitely generated module $M$ and any $d\gg 0$, the minimal free resolution of the truncated module $M_{\geq d}$ is linear, in the sense that it is of the form $\LL(N)$ for some $E$-module $N$~\cite[\S4]{EFS}.  A similar fact plays a key role in Eisenbud-Erman-Schreyer's construction for products of projective spaces~\cite[Proof of Corollary 1.14]{EES}.  However, this fundamentally fails to hold for more general toric varieties, motivating the need for a new approach to constructing Tate resolutions. For instance, let $S$ be the Cox ring of a Hirzebruch surface of type 3, with notation as in Example \ref{ex:RR hirz}. Consider $M = S/(x_3 - x_0^3x_1, x_2)$.  For any $d=(d_1,d_2)$ with $d_1,d_2\geq 0$, the minimal free resolution of $M_{\geq d}$ is a twist of the Koszul complex on $x_3 - x_0^3x_1$ and $x_2$, which is not of the form $\LL(N)$ for any $E$-module $N$. 
%As an aside, we note that free resolutions of sufficiently high truncations of multigraded modules do exhibit a weaker linearity property; we explore this idea in \cite{bounds}.\daniel{I feel like this last sentence muddies the focus of the remark.}
\end{remark}

%%%%%%%%%%%%%%%%
\section{Tate resolutions on toric stacks}
\label{sec:stacks}
%%%%%%%%%%%%%%%%
In this section, we extend the notion of a Tate resolution from $\PP^n$ to more general toric varieties, and we prove (a generalization of) Theorem~\ref{thm:toric exact intro}.  We recall some background on toric stacks in \S \ref{toricbackground}, prove our main results in \S \ref{mainresults}, examine a connection to Beilinson's resolution of the diagonal in \S \ref{othertorics}, and discuss a number of examples in \S \ref{subsec:tate examples}.

Let us start by briefly recalling Eisenbud-Fl\o ystad-Schreyer's approach to defining the Tate resolution in \cite{EFS} and how it differs from ours. Their definition of the Tate resolution of a sheaf $\F$ over $\PP^n$ is as follows: write $\F = \widetilde{M}$ for some $S$-module $M$. By \cite[Corollary 2.4]{EFS}, when $d$ is at least the Castelnuovo--Mumford regularity $\reg(M)$, there is a quasi-isomorphism $H_d(\RR_{\on{classical}}(M_{\ge d})) \xra{\simeq} \RR_{\on{classical}}(M_{\ge d})$.\footnote{Recall that $\RR_{\on{classical}}$ denotes the functor whose output is a complex of $E$-modules, as opposed to a differential $E$-module.  It thus makes sense to discuss the $d^{\th}$ homology module of  $ \RR_{\on{classical}}(M_{\ge d})$.} Letting $F$ be the minimal free resolution of this $d^{\th}$ homology module, they define the Tate resolution $\Tate(\F)$ to be the mapping cone of the composition $F \xra{\simeq}  \RR_{\on{classical}}(M_{\ge d})$.

Attempting to adapt this recipe to the toric setting raises significant challenges. First, applying $\RR$ to ``high" truncations of modules need not yield a differential module that is quasi-isomorphic to its homology; see Remark \ref{truncations}. While one could use the theory of~\cite{BE} to resolve the entire differential module $\RR(M)$ (in the sense of Definition \ref{freeres}), rather than the homology of $\RR(M_{\ge d})$ for some $d$, this raises further complications for adapting the arguments of~\cite{EFS}, which often make use of the fact that the resolution of $M_{\geq d}$ has the form $\LL(P)$ for some module $P$.
%But there is a more fatal complication: over a projective toric variety $X$ with $\Cl(X) \ncong \Z$, taking a minimal free resolution of $\RR(M)$ doesn't give a good notion of a Tate resolution, i.e. a notion with the properties in Theorem \ref{thm:toric exact intro} from the introduction; see \S \ref{mainresults} below.  
There is also a second, more fundamental complication: over a projective toric variety $X$ with $\Cl(X) \ncong \Z$, taking a minimal free resolution of $\RR(M)$ is a purely algebraic construction that cannot differentiate between toric varieties with the same Cox ring.  Thus, any construction of $\Tate(\widetilde{M})$ from $\RR(M)$ must incorporate features of the irrelevant ideal; this is not an issue on $\PP^n$.  And finally, even if one produced such a construction, verifying the properties in Theorems \ref{thm:toric exact intro} and~\ref{thm:exactness properties} would require nontrivial arguments, as it did in both~\cite{EFS} and \cite{EES}.
%doesn't give a good notion of a Tate resolution, i.e. a notion with the properties in Theorem \ref{thm:toric exact intro} from the introduction; see \S \ref{mainresults} below.  \daniel{Talk.}

We therefore use an approach to Tate resolutions that is fundamentally different from that of~\cite{EFS}.  The central novelty is the introduction of a noncommutative analogue of a Fourier--Mukai transform that we use to construct Tate resolutions in great generality and that easily implies Theorems~\ref{thm:toric exact intro} and \ref{thm:exactness properties}.

%%%%%%%%%%%%%%%
\subsection{Setup}
\label{toricbackground}
%%%%%%%%%%%%%%%
Let $N$ be a finitely generated free abelian group and $\Sigma$ a rational fan inside $N \otimes_\Z \Q$. We will write $X=X_\Sigma$ for the toric variety determined by $\Sigma$.  Let $\Sigma(1) = \{\rho_0, \dots \rho_n\}$ denote the rays of $\Sigma$, and let
$
S = k[x_i \text{ : } \rho_i \in \Sigma(1)]
$
denote the \defi{Cox ring} of $\Sigma$. The Cox ring $S$ is equipped with a natural $\Cl(X)$-grading. The \defi{irrelevant ideal} of $\Sigma$ is the ideal $B= \langle \prod_{\rho_i \nsubseteq \sigma} x_i \text{ : } \sigma \in \Sigma \rangle$ of $S$. The variety $Z:= \Spec(S) \setminus V(B)$ admits an action by a torus $G$.   Taking the stack quotient by this action gives the \defi{toric stack} $\X = \X_\Sigma:= [Z / G].$   
If $X$ is smooth, then $\X=X$ and we need not worry about the stack.  When $X$ is singular, we will generally apply adjectives to $\X$ that could be applied to $X$ or $\Sigma$.  For instance, we say $\X$ is \defi{simplicial} (resp. \defi{projective}) if $X$ is simplicial (resp. projective).   

When $X$ is singular, it turns out that the stack $\X$ is more closely connected to the BGG correspondence, Tate resolutions, and other topics considered in this paper.  This aligns with a common philosophy: sheaves are often better behaved on a stack quotient (like $\X$) than on a GIT quotient (like $X$). There are many examples of this in the toric setting, e.g. ~\cite{AKO, bfk, borisov-horja, C, CK,GS,smith}.

While we do not require our toric stacks to be simplicial in this paper, we note that such stacks have particularly nice properties (see~\cite{BCS}): if $k = \CC$ and $\X$ is simplicial, then $\X$ is Deligne-Mumford, and the canonical map $\pi\colon \X \to X$ exhibits $X$ as a \defi{good moduli scheme}, in the sense of \cite[Section 1.2]{alper}.  This implies, for instance, that $\pi_*$ is exact on quasi-coherent sheaves, allowing cohomology computations to pass from $\X$ to $X$.

\begin{notation}
If $\X$ is a toric stack, then $X$ will denote  the corresponding variety.
\end{notation}
\begin{remark}\label{rmk:cox pairs}
It would be interesting to develop the theory of toric Tate resolutions in greater generality, by considering, for example, the more general toric Deligne-Mumford stacks studied by Borisov-Chen-Smith and Fantechi-Mann-Nironi \cite{BCS, FMN}, the extended stacky fans introduced by Jiang \cite{jiang}, %and appearing in~\cite{}
or the toric Artin stacks of Geraschenko-Satriano~\cite{GS}.  We do not pursue these routes in detail.
\end{remark}

\subsection{Main results}
\label{mainresults}
The goal of this section is to prove the following result, which combines and generalizes Theorems ~\ref{thm:toric exact intro} and~\ref{thm:exactness properties}. In part (4) of the Theorem, the notion of an irrelevant subset $I$ of $\{0, \dots, n\}$ and the notation $E_I$ are as introduced in and above Theorem \ref{thm:exactness properties}.

\begin{thm}\label{thm:toric exact}
Let $\X$ be a projective toric stack, and let $\F$ be a coherent sheaf on $\X$.  
There exists %an object
 $\Tate (\F) \in \DM(E)$ with the following properties:
\begin{enumerate}
	\item  $\Tate(\F)$ is a minimal, exact, free differential $E$-module.
	\item  The Tate resolution encodes the sheaf cohomology groups of $\F$: more precisely, for any $a \in \Cl(X)$ and $j\in \ZZ$ we have
	$
	H^j(X , \F(a)) = \underline{\Hom}_E(k, \Tate (\F))_{(a;-j)}.
	$
	\item Let $B$ denote the irrelevant ideal of $\X$, and let $M$ be a $\Cl(X)$-graded $S$-module such that $\widetilde{M} = \F$. Assume $H^0_B(M) = 0$. The injective map $M \into \bigoplus_{a \in \Cl(X)} H^0(X, \F(a))$ induces an embedding $\RR(M) \into \Tate(\F)$ of differential $E$-modules.
	\item  Let $I \subseteq \{0,\dots,n\}$. If $I$ is irrelevant, then $\Tate(\F)\otimes_E E_I$ is exact.
\end{enumerate}
\end{thm}

We call $\Tate(\F)$ the \defi{Tate resolution} of $\F$, though as we will see, saying ``the" Tate resolution is an abuse of terminology, because our construction is only well-defined up to homotopy equivalence.
We construct toric Tate resolutions via a ``noncommutative Fourier-Mukai transform" involving $\X$ and the exterior algebra $E$. That is, we will define a functor $\Phi_\K\colon \Coh(\X) \to \DM(E)$ via the following diagram:

\begin{equation}
\label{pullpush}
\xymatrix{
&\Coh(\X_E)\ar[rr]^-{-\otimes^{\DM} \mathcal K}&&\DM(\X_E)\ar[rd]^{R\pi_{2*}}&\\
\Coh(\X)\ar[rrrr]^-{\Phi_{\K}}\ar[ru]^{\pi_1^*}&&&& \DM(E).
}
\end{equation}

Let us start by defining the categories in \eqref{pullpush}. Loosely speaking, the category $\Coh(\X_E)$ has objects given by $\cO_\X$-$E$ bimodules that are $\Cl(X) \oplus \Z$-graded as $E$-modules and coherent as $\cO_\X$-modules, and the category $\DM(\X_E)$ has objects given by $\cO_\X$-$E$ bimodules that are equipped with a $\Cl(X) \oplus \Z$-grading as $E$-modules and a square 0 endomorphism of degree $(0;-1)$ with respect to the grading on $E$. More precisely, these categories are defined as follows. We note first that, since $\X$ is the quotient stack $[Z/G]$, where $Z$ and $G$ are as in \S \ref{toricbackground}, $\OO_\X$-modules may be identified with $G$-equivariant $\OO_Z$-modules. 

\begin{defn}
A $G$-equivariant $\OO_Z$-module $\E$ is an \defi{$\OO_\X$-$E$ bimodule} if, for each open set $U \subseteq Z$, $\Gamma(U, \E)$ has the structure of a $\Cl(X) \oplus \Cl(X) \oplus \Z$-graded $S$-$E$ bimodule such that
\begin{enumerate}
\item the first $\Cl(X)$-grading on $\Gamma(U, \E)$ is associated to its $S$-module structure, and the complementary $\Cl(X) \oplus \Z$-grading makes it a graded $E$-module;
\item the restriction maps for $\E$ are morphisms of bimodules; and
\item the $E$-actions on each $\Gamma(U, \E)$ are $G$-equivariant, which in this case just means that the actions of the exterior variables $e_i$ on each $\Gamma(U, \E)$ are homogeneous of degree 0 with respect to the $\Cl(X)$-grading associated to the $S$-module structure.
\end{enumerate}
We say an $\OO_\X$-$E$ bimodule is \defi{coherent} if it is so as an $\OO_Z$-module; $\Coh(\X_E)$ is the category of coherent $\OO_\X$-$E$ bimodules. A \defi{differential $\OO_\X$-$E$ bimodule} is an $\OO_\X$-$E$-bimodule equipped with a square zero $\OO_Z$-linear endomorphism that is $G$-equivariant (i.e. homogeneous of degree 0 with respect to the $\Cl(X)$-grading associated to the $S$-module structure) and that makes each $\Gamma(U, \E)$  a degree $(0; -1)$ differential $E$-module, i.e. an object in $\DM(E)$. $\DM(\X_E)$ is the category of differential $\OO_\X$-$E$ bimodules.
\end{defn}

As for the functors in Diagram (\ref{pullpush}): $\pi_1^*$ sends a sheaf $\F \in \Coh(\X)$ to $\F \otimes_k E \in \Coh(\X_E)$. The object $\K \in \DM(\X_E)$ has underlying $G$-equivariant $\OO_{Z}$-module given by $\bigoplus_{a \in \Cl(X)} \OO(a) \otimes_k \om_E(-a; 0)$ and differential given by multiplication on the left by $\sum_{i = 0}^n x_i \otimes e_i$. The functor $R\pi_{2*}$ is defined as follows. Let $\D$ be an object in $\DM(\X_E)$ with differential $\del$. Here is the basic idea: first, we take the derived global sections of $\D$, thought of as a 1-periodic complex of $G$-equivariant $\OO_Z$-modules. We then observe that the resulting object is a 1-periodic complex of $E$-modules and therefore determines an object in $\DM(E)$. Let us now make this precise. Let $\fC_\D^\bullet$ be the \v{C}ech resolution of the $G$-equivariant $\OO_Z$-module $\D$ with respect to the $G$-invariant affine open cover of $Z$ described in \cite[Proposition 4.3]{BCS}. Recall that $\fC_\D^\bullet$ is a complex of $\Cl(X) \oplus \Cl(X) \oplus \Z$-graded $S$-$E$-bimodules. Let $\fC_\D^{G,\bullet}$ denote the $G$-invariant subcomplex of $\fC_\D^\bullet$, i.e. the strand of $\fC_\D^\bullet$ that lies in degrees of the form $(0, a; i) \in \Cl(X) \oplus \Cl(X) \oplus \Z$. We observe that $\fC_\D^{G,\bullet}$ computes the derived global sections of the $G$-equivariant $\OO_Z$-module $\D$. Consider the bicomplex 
$$
\xymatrix{
\cdots & \ar[l]   \fC_\D^{G,0}(0;-1)  \ar[d]^-{d_{\text{\v{C}ech}}}  & \ar[l]_-{\del} \fC_\D^{G,0}  \ar[d]^-{d_{\text{\v{C}ech}}}  & \ar[l]_-{\del}  \fC_\D^{G,0}(0; 1) \ar[d]^-{d_{\text{\v{C}ech}}} & \ar[l] \cdots\\
\cdots  & \ar[l] \fC_\D^{G,1}(0;-1)  \ar[d]^-{d_{\text{\v{C}ech}}}  &\ar[l]_-{-\del} \fC_\D^{G,1}\ar[d]^-{d_{\text{\v{C}ech}}}   &  \ar[l]_-{-\del} \fC_\D^{G,1}(0; 1)\ar[d]^-{d_{\text{\v{C}ech}}} &  \ar[l]  \cdots \\
& \vdots & \vdots  &\vdots &\\
}
$$
Since the $E$-action on $\fC_\D^\bullet$ is $G$-equivariant, this is a bicomplex of $E$-modules. We define $R\pi_{2*}(\D)$ to be the object in $\DM(E)$ obtained by totalizing this bicomplex and then applying the equivalence $\Comper(E) \xra{\simeq} \DM(E)$ from Remark \ref{dgmodules}.

In summary: given $\F \in \Coh(\X)$, $\Phi_\K(\F)$ is given by totalizing the bicomplex

\begin{equation}
\label{cechbicomplex}
\xymatrix{
\cdots  &\ar[l]  \bigoplus_{a \in \Cl(X)} \fC^{G,0}_{\F(a)} \otimes_k \om_E(-a; -1)\ar[d]^-{d_{\text{\v{C}ech}} \otimes 1}  &&  \ar[ll]_-{\sum_{i = 0}^n x_i \otimes e_i} \bigoplus_{a \in \Cl(X)} \fC^{G,0}_{\F(a)} \otimes_k \om_E(-a; 0)  \ar[d]^-{d_{\text{\v{C}ech}} \otimes 1} & \ar[l] \cdots\\
\cdots   &\ar[l]\bigoplus_{a \in \Cl(X)} \fC^{G,1}_{\F(a)} \otimes_k \om_E(-a; -1)  \ar[d]^-{d_{\text{\v{C}ech}} \otimes 1}  &&\ar[ll]_-{-\sum_{i = 0}^n x_i \otimes e_i} \bigoplus_{a \in \Cl(X)} \fC^{G,1}_{\F(a)} \otimes_k \om_E(-a; 0) \ar[d]^-{d_{\text{\v{C}ech}} \otimes 1}  &  \ar[l] \cdots \\
& \vdots && \vdots  &\\
}
\end{equation}
and then applying the equivalence $\Comper(E) \xra{\simeq} \DM(E)$ from Remark \ref{dgmodules}.

\begin{proof}[Proof of Theorem \ref{thm:toric exact}]
Let us write the terms of the bicomplex (\ref{cechbicomplex}) as $Y_{i, j}$. The columns of this bicomplex split $E$-linearly; choose an $E$-linear decomposition $Y_{i,j} = B_{i,j} \oplus H_{i,j} \oplus L_{i,j}$ for each $i, j$ such that $B_{i,j} \oplus H_{i,j} = Z^{\on{vert}}_{i,j}$, where $Z^{\on{vert}}_{i,j}$ denotes the vertical cycles in $Y_{i,j}$. Observe that there is an isomorphism $H_{i,j} \cong \bigoplus_{a \in \Cl(X)} H^{-j}(\X,\F(a)) \otimes_k \om_E(-a;i)$. Let $\sigma_H\colon Y_{\bullet, \bullet} \to H_{\bullet, \bullet}$ and $\sigma_B\colon Y_{\bullet, \bullet} \to B_{\bullet, \bullet}$ denote the projections, $g\colon L_{\bullet, \bullet} \xra{\cong} B_{\bullet ,\bullet -1}$ the isomorphism induced by the vertical differential, and $\pi = g^{-1} \sigma_B$. By \cite[Lemma 3.5]{EFS}, the object $\Phi_\K(\F) \in \DM(E)$ is homotopy equivalent to the differential $E$-module $\Tate(\F)$ with underlying module $\bigoplus_{i = 0}^{\dim{X}} \bigoplus_{a \in \Cl(X)} H^i(\X, \F(a)) \otimes_k \om_E(-a;i)$ and differential given by
\[ 
\del_{\Tate(\F)} = \sum_{j \ge 0} \sigma_H(d_{\on{hor}} \pi)^j d_{\on{hor}}, 
\]
where $d_{\on{hor}}$ is the horizontal differential in the bicomplex \eqref{cechbicomplex}. We make the following observations about the differential $\del_{\Tate(\F)}$:
\begin{enumerate}
\item[(i)] The $j = 0$ term of $\del_{\Tate(\F)}$ applied to an element in $H^i(\X, \F(a)) \otimes_k \om_E(-a;i)$ is given by multiplication on the left by $(-1)^i \sum_{i = 0}^n x_i \otimes e_i$.
\item[(ii)] The $j > 0$ terms of $\del_{\Tate(\F)}$ decrease the sheaf cohomology degree. That is, given $z \in H^i(\X, \F(a)) \otimes_k \om_E(-a;i)$, the $j > 0$ terms of $\del_{\Tate(\F)}$ map $z$ to $\bigoplus_{\ell < i} H^\ell(\X, \F(a)) \otimes_k \om_E(-a;\ell-1)$. 
\item[(iii)] Combining (i) and (ii), we conclude that, given $z \in H^0(\X, \F(a)) \otimes_k \om_E(-a ; 0)$, we have $\del_{\Tate(\F)}(z) = (\sum_{i = 0}^n x_i \otimes e_i) \cdot z$. 
\end{enumerate} 

Since the rows of \eqref{cechbicomplex} are exact, and its columns are bounded, the differential module $\Phi_\K(\F)$ is exact, and so $\Tate(\F)$ is as well. The underlying $E$-module of $\Tate(\F)$ is free, and the minimality of $d_{\on{hor}}$ as a morphism of $E$-modules implies the minimality of $\del_{\Tate(\F)}$. This proves (1). Part (2) is immediate from our description of the underlying module of $\Tate(\F)$. It is also clear that the embedding $M \into \bigoplus_{a \in \Cl(X)} H^0(X, \F(a))$ induces an embedding $\RR(M) \into \Tate(\F)$ on underlying $E$-modules, and (iii) above implies that this is a morphism of differential $E$-modules; this proves (3). As for (4): let $I \subseteq \{0, \dots, n\}$ be irrelevant, and let $\K_I := \K \otimes_E E_I$. The homology of $\K_I$ is just a direct sum of twists of the homology of the Koszul complex on the variables in $\{x_i \text{ : } i \in I\}$, so $\K_I$ is exact. It follows that, if we tensor the bicomplex (\ref{cechbicomplex}) with $E_I$, the result has exact rows. Since the columns of this bicomplex are bounded, its totalization is therefore exact, i.e. $\Phi_\K(\E) \otimes_E E_I$ is exact. Since $\Phi_\K(\F) \otimes_E E_I$ is homotopy equivalent to $\Tate(\F) \otimes_E E_I$, we conclude that $\Tate(\F) \otimes_E E_I$ is exact.
\end{proof}

\begin{remark}\label{rmk:categorified}
Let $\F \in \Coh(\X)$. We may identify the underlying module of $\Tate(\F)$ with $\bigoplus_{i = 0}^{\dim X} T_i$, where $T_i = \bigoplus_{a \in \Cl(X)} H^i(\X, \F(a)) \otimes_k \om_E(-a, i)$. By degree considerations, each $T_{\le i} = \bigoplus_{j \le i} T_j$ is a differential submodule of $\Tate(\F)$, yielding a filtration
%and so $\Tate(\F)$ may be equipped with a filtration
$$
T_0 \subseteq T_{\le 1} \subseteq \cdots \subseteq T_{\le \dim(X)} = \Tate(\F).
$$
\end{remark}

\begin{remark}\label{rmk:intro Mori}
Theorem \ref{thm:toric exact} has a subtle implication.  A theme in research on syzygies is that the geometry of $\PP^n$ can influence the algebraic properties of graded free resolutions.
 %(see~\cite{ES-ICM} for a dramatic such example)
%\daniel{Is this parenthetical out of place?} \michael{I don't hate it but I'd be okay removing it.}.  
In the toric situation, where varying the irrelevant ideal can yield different toric varieties/stacks,  one can ask which of these geometric objects affects the algebraic properties of modules over the Cox ring?  Of course, the answer is: all of them.  
For example, if $X$ and $X’$ have the same Cox ring $S$, and if $M$ is an $S$-module, then Theorem \ref{thm:toric exact}(3) shows that the differential $E$-module $\RR(M)$ can be completed to a Tate resolution in distinct ways, each of which imposes constraints on the homology of $\RR(M)$.   Via Proposition~\ref{tor}, this shows that the Betti numbers of $M$ are influenced by both $X$ and $X'$ and, more generally, by the toric varieties/stacks arising in the Mori chamber decomposition of the effective cone in $X$.
% can influence the properties of $S$-modules.  %Remark~\ref{ex:RRI hirz} provides an example on this theme (via Theorem 6.11).
\end{remark}

As one can see from the proof of Theorem \ref{thm:toric exact}, our construction of the Tate resolution is, \emph{a priori}, only well-defined up to homotopy equivalence: the differential we get depends on the choices of splittings of the columns of the bicomplex (\ref{cechbicomplex}) that we use in our application of \cite[Lemma 3.5]{EFS}. However, Theorems \ref{weightedthm} and \ref{comparison} below imply that our construction of the Tate resolution is well-defined up to isomorphism in some special cases. In particular, we will see that our Tate resolutions agree with those of~\cite{EFS,EES}.

We say a toric stack is a \defi{generalized weighted projective stack} if its associated toric variety is a fake weighted projective space, i.e. a quotient of a weighted projective space by the action of a finite abelian group. We remind the reader that we only consider weighted projective spaces whose weights are pairwise relatively prime. The divisor class group of a fake weighted projective space is isomorphic to $\Z \oplus A$ for some finite abelian group $A$.

\begin{thm}
\label{weightedthm}
Suppose $\X$ is a generalized weighted projective stack, and let $\F$ be a coherent sheaf on $\X$.  Choose a $B$-saturated $S$-module $M$ such that $\widetilde{M} = \F$. Any differential module $T$ with Properties (1) - (4) in Theorem \ref{thm:toric exact} is isomorphic to $\cone(F \xra{\simeq} \RR(M))$, where $F$ is the minimal free resolution of $\RR(M)$, in the sense of Definition \ref{freeres}. In particular, Properties (1) - (4) in Theorem \ref{thm:toric exact} determine the Tate resolution of $\F$ up to isomorphism. 
\end{thm}

%\begin{remark}
%It is known that any simplicial projective toric variety $X$ such that $\Cl(X) = \Z$ is a fake weighted projective space, i.e. a quotient of a weighted projective space by the action of a finite abelian group. It follows that 
%\end{remark}

\begin{proof}
Consider the embedding $\RR(M) \into T$ in Theorem \ref{thm:toric exact}(3). Bearing part (2) of Theorem~\ref{thm:toric exact} in mind, we conclude that, as a map of $E$-modules, this embedding may be interpreted as the inclusion $\RR(M) \into  \RR(M) \oplus N$, where $N = \bigoplus_{i = 1}^{\dim X} \bigoplus_{a \in \Cl(X)} H^i(\X, \F(a)) \otimes_k \om_E(-a; i)$. By degree considerations, we see that, with respect to this decomposition, the differential on $T$ has the form
$
\begin{pmatrix}
\del_{\RR} & \a \\
 0 & \b
\end{pmatrix},
$
and so $T$ is isomorphic to the mapping cone of a morphism $\a : (N, -\b) \to \RR(M)$ of differential modules. Notice that $(N, -\b)$ is a minimal, free differential $E$-module. To conclude that $(N, -\b)$ is the minimal free resolution of $\RR(M)$, in the sense of Definition \ref{freeres}, it suffices to show that $(N, -\b)$ is a free flag. Let $\pi$ be a surjection $\Cl(X) \to \Z$; note that $\ker(\pi)$ is a finite abelian group. For $j \in \Z$, define
$$
N_{j} = \bigoplus_{i = 1}^{\dim X} \bigoplus_{\pi(a) \ge -j} H^i(\X, \F(a)) \otimes_k \om_E(-a;i).
$$
It follows from an analogue of the Serre Vanishing Theorem that $N_j = 0$ for $j \ll 0$.  By the minimality of $T$, the differential maps each $N_j$ to $\bigoplus_{\l < j} N_\l$. Thus,  $N = \bigoplus N_j$ is a free flag differential module. The last statement follows from uniqueness of minimal free resolutions of differential modules (see Theorem \ref{EU}).
\end{proof}

\begin{cor}
\label{pncor}
When $\X = \PP^n$, we have $\Tate = \Fold \circ \Tate_{\on{EFS}}$, where $\Tate_{\on{EFS}}$ denotes Eisenbud-Fl\o ystad-Schreyer's Tate resolution functor, and the functor $\Fold$ is as in Definition \ref{defn:fold}.
\end{cor}

\begin{proof}
Clear from Theorem \ref{weightedthm} and the relation $\RR = \Fold \circ \RR_{\on{classical}}$ from Remarks \ref{RRbicomplex}(2).
\end{proof}

\begin{remark}
\label{truncations}
Suppose $\X = \PP^n$. As discussed in the beginning of this section, if $M$ is a finitely generated $S$-module, then $\RR(M_{\ge d})$ is quasi-isomorphic to its homology when $d \ge \on{reg}(M)$. Thus, the Tate resolution of $\widetilde{M}$ can be defined by resolving $H(\RR(M_{\ge d}))$, rather than all of $\RR(M_{\ge d})$; this is the approach taken in \cite{EFS} and also \cite{EES} over multiprojective spaces. Unfortunately, this does not work more generally; for instance, one can check that this fails in the example discussed in Remark \ref{rmk:hirz non linear and qiso}, over a Hirzebruch surface of type 3. 
\end{remark}

%Theorem \ref{weightedthm} does not extend beyond the case 
It is not always the case that $\cone(F \xra{\simeq} \RR(M))$, where $F$ is the minimal free resolution of $\RR(M)$, yields the Tate resolution.  For instance, if $X=\PP^1 \times \PP^1$, then that mapping cone is the corner complex as defined in~\cite[\S3]{EES}.  
However, we do have:

\begin{thm}
\label{comparison}
When $\X$ is a product of projective spaces, any differential module $T$ with Properties (1) - (4) in Theorem \ref{thm:toric exact} satisfies $T = \Fold \circ \Tate_{\on{EES}}$, where the functor $\Fold$ is as introduced in Definition \ref{defn:fold}, and $\Tate_{\on{EES}}$ denotes Eisenbud-Erman-Schreyer's Tate resolution functor. In particular, Properties (1) - (4) in Theorem \ref{thm:toric exact} determine the Tate resolution up to isomorphism in this case. \end{thm}

\begin{proof}
The differential module $T$ can be realized as the folding of a complex of $E$-modules, as in Definition~\ref{defn:fold}. We may therefore assume that $T$ is a complex. The exactness properties--that is, Theorem~\ref{thm:toric exact}(4)--imply that $T$ has exact strands.  By \cite[Theorem~3.3]{EES}, this implies that every corner complex of $T$ is also exact.  Let $T_{d-\text{tail}} \to T_{d-\text{head}}$ be the degree $d$ corner complex of $T$, as defined in~\cite[Section 3]{EES}.   Theorem~\ref{thm:toric exact}(2) and (3) imply that, for $d\gg 0$,  $T_{d-\text{head}}$ is precisely $\RR(M_{\geq d})$.  Exactness of the corner complex and minimality of $T$ imply that $T_{d-\text{tail}}$ is the minimal free resolution of $\RR (M_{\geq d})$ and is thus uniquely determined.  It follows that $T_{d-\text{tail}}$ is isomorphic to the analogous degree $d$ tail of $\Tate (\cF)$ for any $d\gg 0$; we conclude that $T$ and $\Tate (\cF)$ are isomorphic complexes.
\end{proof}

While the toric stacks considered in Theorems \ref{weightedthm} and \ref{comparison} are relatively simple, we believe that similar results should hold much more broadly; see Conjecture \ref{c:intrinsic} below. But carrying out such a generalization will require new homological methods.  The proofs of Theorems~\ref{weightedthm} and \ref{comparison} rely on the fact that we can construct $\Tate (\F)$ from $\RR (M)$ by taking minimal free resolutions of complexes or differential modules in these cases.  However, this approach makes no explicit reference to the irrelevant ideal of $X$. By considering situations where one has two toric varieties $X$ and $X'$ with the same Cox ring but with different irrelevant ideals, one can see that characterizing the Tate resolution up to isomorphism in general will require an approach that makes greater use of the exactness properties of Theorem~\ref{thm:toric exact}(3). To underscore the algebraic challenge: a free resolution yields an object that is exact, but how does one produce an object with all of the exactness properties from~Theorem~\ref{thm:toric exact}(3)?

%%%%%%%%%%%%%%%%%%%%%%%%%%%%%%%%%%%%%%%%%%%
\subsection{Tate resolutions over $\PP^n$ and Beilinson's resolution of the diagonal}
\label{othertorics}
%%%%%%%%%%%%%%%%%%%%%%%%%%%%%%%%%%%%%%%%%%%
A major emphasis of \cite{EFS} is the relationship between Tate resolutions and Beilinson monads
%resolution of the diagonal over 
on $\PP^n$. Our ``noncommutative Fourier-Mukai" construction of Tate resolutions over $\PP^n$ takes this relationship a step further. The goal of this subsection is to explain this, and in doing so provide some intuition for our construction of toric Tate resolutions. 

Let $\cB$ denote Beilinson's resolution of $\OO_\Delta$ in $\Coh(\PP^n \times \PP^n)$, which has the form:
$$
0 \from \OO \boxtimes \OO \from \OO(-1) \boxtimes \Omega^{1}(1) \from \cdots \from \OO(-n) \boxtimes \Omega^n(n) \from 0;
$$
it is a Koszul complex on the section $\sum_{i = 0}^n x_i \boxtimes \frac{\partial}{\partial y_i}$ of $\OO_{\PP^n}(1) \boxtimes T_{\PP^n}(-1))$.
Let $\Phi_{\cB}$ denote the Fourier-Mukai transform given by $\F \mapsto R\pi_{2*}(\pi_1^*( \F) \otimes \cB)$.   Since $\cB$ is a resolution of the diagonal, $\Phi_{\cB}(\F)$ is isomorphic to $\F$ in $\Db(\PP^n)$.  From~\cite{beilinson}, there is a natural ``Beilinson spectral sequence'' whose $E^1$ terms have the form $H^i(\PP^n,\F(-j)) \otimes \Omega^j_{\PP^n}(j)$.  But one can go further and produce an actual complex of sheaves
%(as opposed to just an in the derived category)
whose terms are sums of these sheaves and whose $0^{\th}$ homology is $\F$. Such a complex is called a \defi{Beilinson monad} for $\F$; see~\cite{AO, EFS, kapranov} for details.

To explicitly construct the Beilinson monad, Eisenbud-Fl\o ystad-Schreyer define a functor $\mathbf{\Omega}$ which sends a free complex of $E$-modules\footnote{In \cite{EFS}, the exterior algebra has the $\ZZ$-grading where $\deg(e_i)=-1$ for all $i$.} to a complex of $\cO_{\PP^n}$-modules \cite[\S 6]{EFS}.  The functor is defined by the formula $\mathbf{\Omega}(\omega_E(j))=\Omega_{\PP^n}^j(j)$, with the key point being that there is a canonical map $\Hom_E(\om_E(i), \om_E(j)) \to \Hom_{\PP^n}(\Omega_{\PP^n}^i(i), \Omega_{\PP^n}^j(j))$.  They prove that applying $\mathbf{\Omega}$ to the Tate resolution of $\F$ yields a Beilinson monad for $\F$.
%In particular, Tate resolutions allow one to construct a Beilinson monad without the need for a spectral sequence.
%Specifically it is a monad for $\F$ in the sense that
%%: \Com_{\on{Free}}(E) \to \Com(\PP^n)$,
%%and they prove in \cite[Theorem 6.1]{EFS} that, for any sheaf $\F$ on $\PP^n$, one has
%$$
%H_i(\Omega \Tate_{\on{EFS}}(\F)) \cong \begin{cases} \F, & i = 0; \\ 0, & i \ne 0, \\ \end{cases}
%$$
%where, $\Tate_{\on{EFS}}$ denotes the Tate resolution functor defined in \cite{EFS}, as in Corollary \ref{pncor}; and the terms equal the terms from the Beilinson spectral sequence.
%The functor $\Omega$ is defined by sending $\omega_E(j)$ to $\Omega_{\PP^n}^j(j)$, with the key point being that there is a canonical map $\Hom_E(\om_E(i), \om_E(j)) \to \Hom_{\PP^n}(\Omega_{\PP^n}^i(i), \Omega_{\PP^n}^j(j))$ (see  \cite[Proposition 5.6]{EFS}) allowing one to define the differentials. 

Our construction demonstrates that this connection can be ``lifted'' one level further.  
Indeed, the object $\K \in \DM(\PP^n_E)$ is equivalent, via the $\Fold$ functor (Definition \ref{defn:fold}), to:
% the doubly infinite complex
\begin{equation}
\label{Kst}
\cdots \gets \OO_{\PP^n} \otimes_k \om_E \gets \OO_{\PP^n}(1) \otimes_k \om_E(-1) \gets \cdots \gets \OO_{\PP^n}(-n) \otimes_k \omega_E(-n) \gets \cdots 
%\cdots \to \OO_{\PP^n}(-1) \otimes_k \om_E(1) \to \OO_{\PP^n} \otimes_k \om_E \to \OO_{\PP^n}(1) \otimes_k \om_E(-1) \to \cdots
\end{equation}
%Observe what happens if one applies the $\mathbf{\Omega}$ functor to the righthand summands of this doubly infinite complex.  
Since $\mathbf{\Omega}(\omega_E(i))=0$ if $i>n$ or $i<0$, replacing each $\om_E(i)$ with $\mathbf{\Omega}(\om_E(i))$ transforms the doubly infinite complex \eqref{Kst} into a bounded one whose terms are identical to those of Beilinson's resolution. The intuition is that applying ``$\id\times \mathbf{\Omega}$" to \eqref{Kst} recovers Beilinson's resolution.\footnote{As an aside, we note that one may also think of the bicomplex \eqref{cechbicomplex} above as a ``lift from $\PP^n$ to $E$" of the bicomplex that gives the Beilinson spectral sequence.} %We make this idea precise, and generalize it, in Proposition \ref{1U}.
The situation may be summarized by the following diagram:
\begin{equation}
\label{keydiagram}
\xymatrix{
K \ar[rr]^-{}\ar[d]^{\rho}&& \mathcal K\ar[d]^{R\pi_{2*}}\ar[rr]^-{``\id \times \mathbf{\Omega}"}&&   \text{Beilinson's resolution of $\OO_\Delta$} \ar[d]^-{R\pi_{2*}}\\
\RR(S) \ar[rr]^-{\text{ inclusion }} && \Tate(\cO_{\PP^n}) \ar[rr]^-{ \mathbf{\Omega}} && \text{Beilinson monad of $\OO_{\PP^n}$}.
}
\end{equation}
Here, $K$ denotes the Koszul complex on the variables in $S$, considered as an $S$-$E$-bimodule; the map $\rho$ denotes restriction of scalars along the inclusion $E \into S \otimes_k E$; and the map $K \to \K$ is given by sheafification and inclusion. The bottom row follows from \cite[Theorem 6.1]{EFS}, and our construction of the Tate resolution provides the rest of the diagram. %construction of the Tate resolution provides the rest of the diagram. 
So, while \cite{EFS} use a Tate resolution to recover the Beilinson monad, our construction directly connects to Beilinson's resolution itself.  Theorem \ref{BMdiagram} and Proposition \ref{1U} below make the ideas conveyed by this diagram precise and extend them to any projective toric stack.

\subsection{Examples of toric Tate resolutions}\label{subsec:tate examples}
%%%%%%%%%%%%%%%%%%%%%%%%%%

The exactness properties in Theorem \ref{thm:toric exact}(4) underscore a key point: if $\F$ is a sheaf on $\X$, then $\Tate (\F)$ depends not just on the Cox ring of $\X$ but on its irrelevant ideal. 
The following examples are intended to illustrate this point.
%how the exactness properties of Tate resolutions encode the geometry of the irrelevant ideal.
%The following examples are intended to illustrate how the exactness properties of Tate resolutions encode the geometry of the irrelevant ideal.
\begin{example}
Let $\X=\PP(1,1,2)$. In Example \ref{weightedexample}, we computed $\RR(S)$; let us now compute $\Tate(\OO_\X)$. By Theorem \ref{weightedthm}, we need only compute the minimal free resolution $F$ of $\RR(S)$ and take the mapping cone of the augmentation $F \to \RR(S)$. The result looks like:
\footnotesize
\[
\xymatrix{
\cdots\ar@/_-1.5pc/[rr]\ar[r]&\om_E(5;2)^{\oplus 2}\ar[r]&\om_E(4;2) \ar[r]\ar@/_-1.5pc/[rrrr]^-{e_0e_1e_2}& 0\ar[r]& 0\ar[r]& 0\ar[r]&\om_E \ar[r]\ar@/_-1.5pc/[rr]& \om_E(-1;0)^{\oplus 2}\ar[r]
%\ar@/_1.5pc/[rr] 
&\cdots
% \om_E(-2;0)^{\oplus 4}\ar[r]&\dots.
}
\]
\normalsize
\vskip\baselineskip
\noindent In detail: consider the $S$-module $M = \bigoplus_{ a \in \Z} H^2(\X, \OO_{\X}(a))$ as a complex concentrated in homological degree $-2$. The minimal free resolution $F$ of $\RR(S)$ is given by $\RR(M)$. The augmentation is given by a single map from $F$ to $\RR(S)$ that sends $\om_E(4;2)$ to $\om_E(0;0)$ via multiplication by $e_0e_1e_2$.  Note, for instance, that the $\om_E(5;2)^{\oplus 2}$ term encodes the fact that $H^2(\X,\cO_{\X}(-5)) = k^2$.
\end{example}

\begin{example}
Continuing with $\X=\PP(1,1,2)$, let $C\subseteq \X$ be the genus 1 curve $C=V(x_0^4+x_1^4+x_2^2)$.
% and $M=S/(x_0^4 + x_1^4 + x_2^2)$ the coordinate ring of $C$.  Recall from Example \ref{weightedexample} that $\RR(M)$ has the form
%\footnotesize
%\[
%\RR(M) = \left[\xymatrix{
%\om_E \ar[r]^-{}\ar@/_2pc/[rr]_-{}& \om_E(-1;0)^{\oplus 2} \ar[r]^-{}\ar@/_-2pc/[rr]^-{}&\om_E(-2;0)^{\oplus 4} \ar[r]^-{}\ar@/_2pc/[rr]_-{} &\om_E(-3;0)^{\oplus 6}\ar[r]^-{} & \cdots
%}\right]
%\]
%\normalsize
%To compute the minimal free resolution of $\RR(M)$, we can apply~\cite[Construction 2.8]{BE} and ``kill cycles'' degree by degree.
The Tate resolution $\Tate(\OO_C)$ has the form:
%giving the differential module
%\footnotesize
%\[
%\xymatrix{
%\dots \ar[r] \ar@/_-2pc/[rr]_-{}  & \om_E(3;1)^{\oplus 6}  \ar[r] \ar[r]^-{}\ar@/_2pc/[rr]_-{}& \om_E(2;1)^{\oplus 4} \ar[r]^-{}\ar@/_-2pc/[rr]^-{}&\om_E(1;1)^{\oplus 2}\ar[r]^-{}&\om_E(0;1).
%}
%\]
\footnotesize
\[
\xymatrix{
\cdots \ar[r] \ar@/_2pc/[rr]^-{}& \om_E(2;1)^{\oplus 4} \ar[r]^-{}\ar@/_-2pc/[rr]^-{}&\om_E(1;1)^{\oplus 2}\ar[r]^-{}\ar[r]^-{}\ar@/_2pc/[rr]_-{}&\om_E(0;1)\oplus \om_E(0;0)\ar@/_-2pc/[rr]^-{} \ar[r] & \omega(-1;0)^{\oplus 2}\ar[r] \ar@/_2pc/[rr]^-{}&\omega(-2;0)^{\oplus 4} \ar[r]& \cdots
}
\]
\noindent \normalsize  
The terms from $\RR\left(\frac{S}{x_0^4+x_1^4+x_2^2}\right)$ are those with no twist in the auxiliary degree, i.e. those of the form $\omega_E(i; 0)$ for some $i$.  Note also that the middle term $\om_E(0;1)\oplus \om_E(0;0)$ corresponds to the fact that, because $C$ is a genus one curve, we have $H^1(\X,\cO_C) = k = H^0(\X,\cO_C)$.
\end{example}

\begin{example}
For another example with $\X=\PP(1,1,2)$, let $P$ denote the stacky point $V(x_0,x_1)$.  Of course $\cO_P$ has no higher cohomology groups, and we have
$
H^0(\X, \cO_P(d)) = k$ when $d$ is even and $0$ when $d$ is odd.  The Tate resolution $\Tate(\cO_P)$ thus looks as follows:
\footnotesize
\[
\xymatrix{
\cdots \ar[r]&\om_E(2;0) \ar[r]\ar@/_-1.5pc/[rr]^-{e_2}& 0\ar[r]&\om_E \ar[r]\ar@/_-1.5pc/[rr]^-{e_2}& 0\ar[r]& \om_E(-2;0)\ar[r]\ar@/_-1.5pc/[rr]^-{e_2} &0\ar[r]&\om_E(-4;0)\ar[r] &\dots
}
\]
\normalsize
\end{example}

\begin{example}\label{ex:hirz3 tate}
Let $\X = X$ be the Hirzebruch surface of type $3$, with Cox ring $S$ as described in Example \ref{ex:RR hirz}.  The irrelevant ideal is $B = (x_0, x_2) \cap (x_1, x_3)$. The Tate resolution $\Tate(\cO_{X})$ has summands corresponding to $H^0, H^1$ and $H^2$ groups.  We illustrate these in the pictures below, with the left picture depicting the degrees where the sheaf has $H^0$; the middle picture the $H^1$'s; and the right picture the $H^2$'s:

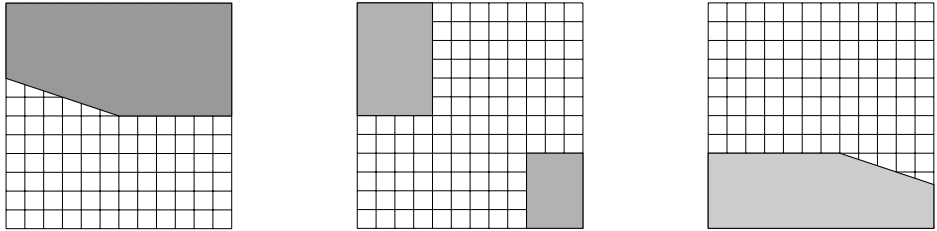
\[
\begin{tikzpicture}[scale = .25]
\draw[step=1cm,black,very thin] (0,0) grid (12,12);
\draw[fill=gray!80] (6,6)--(0,8)--(0,12)--(12,12)--(12,6)--(6,6);
\end{tikzpicture}
\qquad \qquad
\begin{tikzpicture}[scale = .25]
\draw[step=1cm,black,very thin] (0,0) grid (12,12);
\draw[fill=gray!60] (4,6)--(0,6)--(0,12)--(4,12)--(4,6);
\draw[fill=gray!60] (9,4)--(12,4)--(12,0)--(9,0)--(9,4);
\end{tikzpicture}
\qquad \qquad
\begin{tikzpicture}[scale = .25]
\draw[step=1cm,black,very thin] (0,0) grid (12,12);
\draw[fill=gray!40] (7,4)--(10,3)--(12,2.33)--(12,0)--(0,0)--(0,4)--(7,4);
\end{tikzpicture}
\]
Putting these together, we get a picture that looks like

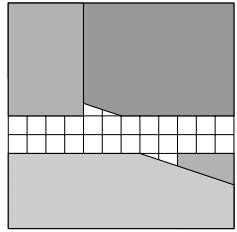
\[
\begin{tikzpicture}[scale=0.25]
\draw[step=1cm,black,very thin] (0,0) grid (12,12);
\draw[fill=gray!80] (6,6)--(0,8)--(0,12)--(12,12)--(12,6)--(6,6);
\draw[fill=gray!60] (4,6)--(0,6)--(0,12)--(4,12)--(4,6);
\draw[fill=gray!60] (9,4)--(12,4)--(12,0)--(9,0)--(9,4);
\draw[fill=gray!40] (7,4)--(10,3)--(12,2.33)--(12,0)--(0,0)--(0,4)--(7,4);
\end{tikzpicture}
\]
To illustrate the Exactness Properties, let us consider the irrelevant subset $I=\{0,2\}$.  Theorem~\ref{thm:exactness properties} implies that the Tate resolution remains exact after modding out by $e_1,e_3$. Since the variables $x_0$ and $x_2$ have degree $(1,0)$, the differential on this restricted Tate resolution will be ``purely horizontal''; that is to say, the differential will send each generator only to elements to the right and in the same row.  
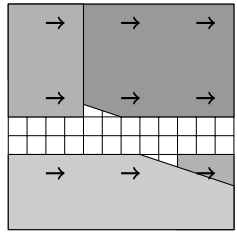
\[
\begin{tikzpicture}[scale=0.25]
\draw[step=1cm,black,very thin] (0,0) grid (12,12);
\draw[fill=gray!80] (6,6)--(0,8)--(0,12)--(12,12)--(12,6)--(6,6);
\draw[fill=gray!60] (4,6)--(0,6)--(0,12)--(4,12)--(4,6);
\draw[fill=gray!60] (9,4)--(12,4)--(12,0)--(9,0)--(9,4);
\draw[fill=gray!40] (7,4)--(10,3)--(12,2.33)--(12,0)--(0,0)--(0,4)--(7,4);
\draw[->,thick] (6,7)--(7,7);
\draw[->,thick] (6,11)--(7,11);
\draw[->,thick] (6,3)--(7,3);
\draw[->,thick] (2,7)--(3,7);
\draw[->,thick] (2,11)--(3,11);
\draw[->,thick] (2,3)--(3,3);
\draw[->,thick] (10,7)--(11,7);
\draw[->,thick] (10,11)--(11,11);
\draw[->,thick] (10,3)--(11,3);
 \end{tikzpicture}
 \]
 The differential module $\Tate(\cO_X)\otimes_E E_{\{0,2\}}$ therefore decomposes as a direct sum of exact differential modules corresponding to the rows of the above picture.

Now consider the  irrelevant subset $I=\{1,3\}$.  Since the variables $x_1$ and $x_3$ have degrees $(-3,1)$ and $(0,1)$, the differential on this restricted Tate resolution will be ``vertical'' and also ``vertical and to the left".  That is, the differential on $\Tate(\cO_X)\otimes_E E_{\{1,3\}}$ looks like this:
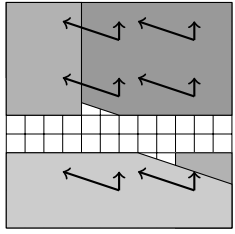
\[
\begin{tikzpicture}[scale=0.25]
\draw[step=1cm,black,very thin] (0,0) grid (12,12);
\draw[fill=gray!80] (6,6)--(0,8)--(0,12)--(12,12)--(12,6)--(6,6);
\draw[fill=gray!60] (4,6)--(0,6)--(0,12)--(4,12)--(4,6);
\draw[fill=gray!60] (9,4)--(12,4)--(12,0)--(9,0)--(9,4);
\draw[fill=gray!40] (7,4)--(10,3)--(12,2.33)--(12,0)--(0,0)--(0,4)--(7,4);
\draw[->,thick] (6,7)--(3,8);
\draw[->,thick] (6,7)--(6,8);
\draw[->,thick] (6,10)--(3,11);
\draw[->,thick] (6,10)--(6,11);
\draw[->,thick] (6,2)--(3,3);
\draw[->,thick] (6,2)--(6,3);
\draw[->,thick] (10,7)--(7,8);
\draw[->,thick] (10,7)--(10,8);
\draw[->,thick] (10,10)--(7,11);
\draw[->,thick] (10,10)--(10,11);
\draw[->,thick] (10,2)--(7,3);
\draw[->,thick] (10,2)--(10,3);
 \end{tikzpicture}
\]
\end{example}

\begin{example}\label{ex:P113 tate}
%Following the discussion in Remark~\ref{rmk:cox pairs}, 
Let $S$ be as in the previous example, but alter the irrelevant ideal so that it is now $B':=(x_1)\cap (x_0,x_2,x_3)$.   While this does not fit into the setup of Theorem~\ref{thm:toric exact}--the irrelevant ideal of a toric variety cannot have a principal minimal prime--our main construction easily extends to this case.  Let $\X'$ be the corresponding stack quotient; the corresponding toric variety $X'$ is the weighted projective space $\PP(1,1,3)$ obtained by contracting the exceptional divisor $V(x_1)$ on the Hirzebruch surface $X$. Since $S$ has remained unchanged from the previous example, $\RR(S)$ is the same as before. However, the Tate resolution $\Tate (\cO_{\X'})$ looks quite different, with no $H^1$ degrees, and $H^0$ and $H^2$ degrees illustrated in dark and light gray as below:
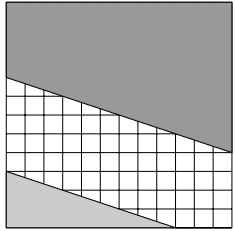
\[
\begin{tikzpicture}[scale = .25]
\draw[step=1cm,black,very thin] (0,0) grid (12,12);
\draw[fill=gray!80] (0,8)--(12,4)--(12,12)--(0,12)--(0,8);
\draw[fill=gray!40] (0,3)--(9,0)--(12,0)--(0,0)--(0,3);
\end{tikzpicture}
\]
This Tate resolution and the one in Example~\ref{ex:hirz3 tate} satisfy different exactness properties.  For instance, consider $I=\{1\}$, which is irrelevant for $\X'$ but not for the Hirzebruch surface.  Since $\deg(x_1)=(-3,1)$, the differential on $\Tate (\cO) \otimes_E E_{\{1\}}$ looks like this:
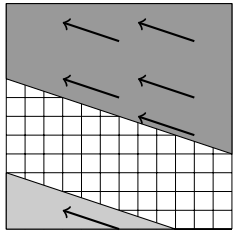
\[
\begin{tikzpicture}[scale = .25]
\draw[step=1cm,black,very thin] (0,0) grid (12,12);
\draw[fill=gray!80] (0,8)--(12,4)--(12,12)--(0,12)--(0,8);
\draw[fill=gray!40] (0,3)--(9,0)--(12,0)--(0,0)--(0,3);
\draw[->,thick] (6,7)--(3,8);
\draw[->,thick] (6,10)--(3,11);
\draw[->,thick] (6,0)--(3,1);
\draw[->,thick] (10,7)--(7,8);
\draw[->,thick] (10,10)--(7,11);
\draw[->,thick] (10,5)--(7,6);
\end{tikzpicture}
\]
The diagonals with slope $-\frac{1}{3}$ yield exact ``strands'' of the Tate resolution $\Tate(\cO_{\X'})$.  Far enough out along any such diagonal in the $H^0$ region, this Tate resolution will agree with $\RR(S)$; in particular, any diagonal of $\RR(S)$ is ``eventually exact''. By contrast, the ``rows'' of $\Tate(\cO_{\X'})$ are not exact, which is to be expected, since $\{0,2\}$ is not an irrelevant subset for $\X'$. 
\end{example}

\begin{remark}\label{rmk:Mori}
Examples~\ref{ex:hirz3 tate} and \ref{ex:P113 tate} point toward a new feature in the multigraded case: when $S$ is the Cox ring of a toric variety $X$, the properties of multigraded $S$-modules are related not just to the geometry of $X$, but also to the other toric varieties that can arise in the Mori chamber decomposition of the effective cone of $X$.  In particular, we used a Tate resolution on the Hirzebruch surface to understand the asymptotic exactness properties of $\RR(S)$ along rows and columns, and we used a Tate resolution for $\PP(1,1,3)$ to understand the exactness properties of $\RR(S)$ along the diagonals of slope $-\frac{1}{3}$.  This is an example of the sort of behavior to which we alluded in Remark~\ref{rmk:intro Mori}.
\end{remark}

%%%%%%%%%%%%%%%%%%
\section{A linear resolution of the diagonal for toric stacks}\label{sec:res diagonal}
%%%%%%%%%%%%%%%%%%

In this section, we construct a Beilinson-type resolution of the diagonal for projective toric stacks. As discussed in \S\ref{othertorics}, Beilinson's resolution of the diagonal over $\PP^n$ arises as a Koszul complex, and our toric version does as well. 

Let $\X$ be a projective toric stack with Cox ring $S= k[x_0, \dots, x_n]$ and associated toric variety $X$. Denote the Cox ring of $\X \times \X$ by $S' = k[x_0, \dots, x_n, y_0, \dots, y_n]$. Let $\Delta\colon \X \to \X \times \X$ be the diagonal morphism, and set $\OO_\Delta := \Delta_*(\OO_\X)$. The naive equations ``$x_i=y_i$'' are not homogeneous and hence do not yield well-defined equations in $S'$. However, we can force these naive equations to be homogeneous using a simple trick.  Let $V \subseteq \Cl(X)$ be the semigroup generated by $\deg(x_0), \dots, \deg(x_n)$, and consider the semigroup ring $R = S'[V]$ as a quotient of the polynomial ring $S'[u_0, \dots, u_n]$,  where each $u_i$ maps to $\deg(x_i) \in V$. Equip $R$ with the $\Cl(X) \oplus \Cl(X)$-grading such that $\deg(u_i) = (\deg(x_i), -\deg(x_i))$.  Each relation
$
x_i - y_iu_i
$
is now homogeneous in $R$, of degree $(\deg(x_i),0)$.  %Our resolution of the diagonal comes from a Koszul complex on the elements $x_i - y_iu_i$:
%Let $S$ denote the Cox ring of $\X$, and \daniel{Do we really need to redefine S?}

\begin{thm}
\label{resdiag}
Let $F$ be the Koszul complex on $x_0 - y_0u_0, \dots, x_n - y_nu_n \in R$. We have:
\begin{enumerate}
\item $F$ is acyclic; that is, $H_i(F) = 0$ for $i > 0$.
\item Viewing $H_0(F)$ as an $S'$-module, we have $\widetilde{H_0(F)} = \OO_\Delta$.
\end{enumerate}
In particular, $\widetilde{F}\in \Coh(\X \times \X)$ is a resolution of $\OO_\Delta$.
\end{thm}

We can write the terms of the complex in Theorem \ref{resdiag} as 
\begin{equation}
\label{rewrite}
F_i = \bigoplus_{a \in V} \bigoplus_{b \in \Cl(X)} (\om_E)_{(b ; i)} \otimes_k S'(-a-b,a).
\end{equation}
Here, as above, $\om_E$ is the $k$-dual of the $\Cl(X) \oplus \Z$-graded exterior algebra $E = \Lambda_k(e_0, \dots, e_n)$. From this view, the differential on $F$ is multiplication on the left by $\sum_{i= 0}^n( e_i \otimes x_i  -  e_i \otimes y_i)$.

Another way of thinking of the complex $F$ is as follows. Let $K$ be the Koszul complex on $x_0, \dots, x_n \in S$. For any $d \in \Cl(X)$, write $K_d$ for the subcomplex of $K$ given by summands of the form $S(-e)$ with $d - e \in V$.
We have
\[
F = \bigoplus_{d \in V} S(-d) \boxtimes K_d(d);
\]
%each $x_i \in S'$ maps the summands to themselves, and the $y_i \in S'$ give maps between the summands. 
with differential $\sum_{i= 0}^n( x_i \otimes e_i  -  y_i \otimes e_i)$, where each $e_i$ shifts terms in the Koszul complex factor, each $y_i \in S'$ maps a summand to itself, and each $x_i \in S'$ maps between summands. 

This latter point of view clarifies the relationship between our resolution and Beilinson's resolution of the diagonal over $\PP^n$ (see \S \ref{othertorics}). When $X = \PP^n$, each complex $\widetilde{K_i(i)}$ of sheaves is a free resolution of $\Omega^i(i)$. In particular, $\widetilde{K_i(i)}$ is exact for $i > n$, and so the summands $S(-i) \boxtimes K_i(i)$ of $F$ with $i > n$ are superfluous; this is an idea to which we will return in \S \ref{sec:finiteres}. %In the case of weighted projective space, we obtain a resolution that is nearly identical to the resolution from~\cite{CK}; see \S \ref{sec:finiteres}.

\begin{remark}
\label{anotherres}
We can give yet another way of constructing the complex $F$ in Theorem \ref{resdiag}, by applying the BGG functor $\LL$ to a module over the $\Cl(X) \oplus \Cl(X) \oplus \Z$-graded exterior algebra $E' = \Lambda_k(e_0, \dots, e_n, f_0, \dots, f_n)$. Let $N$ denote the free $E'$-module $\bigoplus_{d \in V} \om_{E'}(-d, d; 0)$, where $\om_{E'} = \underline{\Hom}_{E'}(E', k)$. It will be helpful to reinterpret $N$ as $\om_{E'}[V]$, just like the ring $R$, except that we declare here that $\deg(u_i) = (\deg(x_i), -\deg(x_i); 0) \in \Cl(X) \oplus \Cl(X) \oplus \Z$. Let $N_\Delta$ be the kernel of the map
$
N \to
%\xra{\begin{pmatrix} f_0 + e_0u_0 \\ \vdots \\ f_n + e_nu_n \end{pmatrix}} 
\bigoplus_{i = 0}^n N(0, -\deg(x_i);-1)
$
whose $i^{\th}$ component is $f_i+e_iu_i$.
In other words, $N_\Delta$ is the $E'[V]$-submodule of $N$ generated by $\g_0 \cdots \g_n$, where $\g_i = x_i - y_iu_i$. One can check that $\LL(N_\Delta) = F$. 
\end{remark}

\begin{example}
\label{colors}
Say $\X = \PP(1, 2)$. Our complex $F$ has the following form:
\footnotesize
$$
\xymatrix{
\vdots & \vdots & \vdots \\
{\bf S'(-3, 3)}  & S'(-5, 3) \oplus S'(-4, 3)  \ar[l]\ar[lu]& S'(-6, 3)\ar[l]\ar[lu] \\
S'(-2, 2) & S'(-4, 2) \oplus {\bf S'(-3, 2)} \ar[l]\ar[lu]\ar[luu]& S'(-5, 2) \ar[l]\ar[lu]\ar[luu]\\
S'(-1, 1)  & {\bf S'(-3, 1)}  \oplus S'(-2, 1) \ar[l]\ar[lu]\ar[luu]& S'(-4, 1)\ar[l]\ar[lu] \ar[luu]\\
S'  & S'(-2, 0) \oplus S'(-1, 0) \ar[l]\ar[lu]\ar[luu]& {\bf S'(-3,0)}\ar[l]\ar[lu]\ar[luu]\\
}
$$
\normalsize
The $x_i$'s map each summand horizontally; $y_0$ maps horizontally and up one position, while $y_1$ maps horizontally and up two positions.  The complexes $S(-d) \boxtimes K_d(d)$ arise among the northwest diagonals; e.g. the bold summands (and corresponding arrows) form that complex for $d=3$. When $d\geq 3$, truncating has no effect, and the strand is simply $S(-d)\boxtimes K(d)$.
\end{example}

%As above, 
%We refer the reader to 
%In our construction, as in Canonaco-Karp's resolution of the diagonal for weighted projective stacks in~\cite{CK}, the sheaves $\Omega_{\PP^n}^i$ from Beilinson's resolution are replaced by truncated Koszul complexes. When $X = \PP^n$, each $\Omega_{\PP^n}^i$ is resolved by a truncated Koszul complex, and so our resolution is closely related, though not identical, to Beilinson's. 

Unlike Beilinson's resolution of the diagonal over $\PP^n$~\cite{beilinson}, or the resolution of the diagonal for a weighted projective stack in~\cite{CK}, the resolution in Theorem \ref{resdiag} has infinite rank terms. However, we show in \S \ref{sec:finiteres} that, in these cases, our resolution can be pruned down to a finite rank complex that is nearly equivalent to those resolutions. 
%In \S \ref{BMsection}, we will use the resolution in Theorem \ref{resdiag} to show that a sheaf on $\X$ can be recovered from its toric Tate resolution (Corollary \ref{monad}). 

%%%%%%%%%%%%%%%%%%%%%%%
\subsection{Proof of Theorem \ref{resdiag}}
\label{proofresdiag}
%%%%%%%%%%%%%%%%%%%%%%%

Denote by $S_\Delta$ the $S'$-module
\[
\bigoplus_{(d,d')\in \Cl(X) \oplus \Cl(X)} H^0(\X\times \X, \cO_{\Delta}(d,d')).
\]
Notice that
$
(S_\Delta)_{(d,d')}  = H^0(\X\times \X, \cO_{\Delta}(d,d'))=H^0(\X, \Delta^*\cO_{\X \times \X}(d,d'))=S_{d+d'}.
$
Equip $S$ with the $\Cl(X) \oplus \Cl(X)$-grading such that $f \in S$ has degree $(0,\on{deg}_S(f))$. The module $S_\Delta$ is a $\Cl(X) \oplus \Cl(X)$-graded $S$-module via the map $S \into S'$ sending $y_i$ to $y_i$.
In fact, for a fixed $d \in \Cl(X)$, $(S_\Delta)_{d,*}$ is a free $\Cl(X) \oplus \Cl(X)$-graded $S$-module with generator in degree $(d,-d)$; it follows that, as an $S$-module, we have
\[
S_\Delta = \bigoplus_{d \in \Cl(X)} S \cdot \epsilon_d, \text{ where $\epsilon_d$ is a generator with $\deg(\epsilon_d)=(d,-d)$}.
\]
%Thus, if $f \in S$ is homogeneous of degree $m \in \Cl(X)$, then $f \cdot \epsilon_d \in S \cdot \epsilon_d$ is homogeneous of degree $(d,-d + m)$. 
Denote by $S_{\Delta}^+$ the submodule of $S_\Delta$ generated by the $\epsilon_d$ such that $d \in V$. In other words,
\[
S_{\Delta}^+ = \bigoplus_{(d,d')\in \Cl(X), \text{ } d \in V} H^0(\X\times \X, \cO_{\Delta}(d,d')).
\]
Since $S_\Delta$ and $S_{\Delta}^+$ agree in high degrees (for instance, they agree in all degrees $(d,d')$ where $d,d' \in V$), they determine the same sheaf on $\X\times \X$, which is to say that $\widetilde{S_{\Delta}^+} = \cO_{\Delta}$.

We can identify multiplication by $y_i$ and $x_i$ on $S_\Delta$ via the relations
\[
x_i\epsilon_d = y_i \epsilon_{d+\deg_S(y_i)}.
\]
We claim that this gives a presentation of $S_\Delta^+$:
\begin{prop}
\label{pres}
$S_\Delta^+$ has generators $\epsilon_d$ for all $d \ge 0$ and relations $x_i\epsilon_d = y_i \epsilon_{d+\deg(x_i)}$ for all $d \ge 0$ and all $i=0,\dots,n$.  In particular, the $S'$-module $S_\Delta^+$ has a free presentation
$
F_1 \xra{\varphi} F_0\to S_\Delta^+ \to  0,
$
where $F_0 = \bigoplus_{d \in V} S'(-d,d)$, and $F_1 = \bigoplus_{i = 0}^n F_0(-\on{deg}_S(x_i),0)$.
\end{prop}
\begin{proof}
We have a surjection $\coker(\varphi) \onto S_\Delta^+$, and the source and target are identical as free $\Cl(X) \oplus \Cl(X)$-graded $S$-modules. Choosing a positive $\Cl(X)$-grading $\theta : \Cl(X) \to \Z$ for $S$ gives a positive $\Cl(X) \oplus \Cl(X)$-grading $\theta' : \Cl(X) \oplus \Cl(X) \to \Z$ of $S$ given by $\theta'(a, b) = a + b$. Since $S_\Delta^+$ is a positively $\Cl(X) \oplus \Cl(X)$-graded $S$-module with respect to $\theta'$, it follows from Nakayama's Lemma that the surjection $\coker(\varphi) \onto S_\Delta^+$ is an isomorphism.
\end{proof}

\begin{proof}[Proof of Theorem \ref{resdiag}]
Observe that $F\otimes S'/(y_0,\dots,y_n)$ is a direct sum of Koszul complexes, one for each $d \ge 0$. In particular, $F\otimes S'/(y_0,\dots,y_n)$ has homology concentrated in degree 0. It follows from Lemma \ref{homnak} and induction that $F$ is exact in positive degrees, i.e. the sequence $x_0 - y_0 u_0, \dots, x_n - y_nu_n$ of elements in $R$ is regular. Now observe that the differential $F_1 \to F_0$ is exactly the presentation map in Proposition \ref{pres}.
\end{proof}

%%%%%%%%%%%%%%%%%%%%%%%%%%%%%
\subsection{Finite rank resolutions of the diagonal}
\label{sec:finiteres}
%%%%%%%%%%%%%%%%%%%%%%%%%%%%%
%The terms of the resolution $F$ in Theorem \ref{resdiag} have infinite rank as $S'$-modules. 
We now show that, when $\X$ is a generalized weighted projective stack (as defined in \S \ref{mainresults}), $F$ has a finite rank subcomplex of length $n$ that is also a resolution of $\OO_\Delta$.  Recall that the divisor class group of the fake weighted projective space $X$ associated to $\X$ is $\Z \oplus A$ for some finite abelian group $A$. As in \S \ref{mainresults}, let $\pi$ be a surjective map $\Cl(X) \to \Z$, and write $w = \sum_{i = 0}^n \pi(\deg(x_i))$. Using the expression of $F$ in \eqref{rewrite}, let $F'$ be the subcomplex of $F$ with terms
$
F'_i = \bigoplus_{\pi(a) \ge 0} \bigoplus_{\pi(b) < w - \pi(a)} (\om_E)_{(b , i)} \otimes_k S'(-a-b,a).
$
Notice that $F'$ is finite rank. 

\begin{example}
When $\X = \PP(1,2)$, $F'$ is the subcomplex spanned by the northwest diagonal strands $ S(-d) \boxtimes K_d(d)$ from Example \ref{colors} for $d<3$.  The key idea is that, when $d \ge 3$, the strand $ S(-d) \boxtimes K_d(d)$ makes an irrelevant contribution to the homology of $F$, and so the cokernel of $F'\to F$ is entirely supported on the irrelevant ideal.
\end{example}

The following Theorem yields a finite rank resolution of the diagonal for a generalized weighted projective stack, giving a slight generalization of a result of Canonaco-Karp for weighted projective stacks \cite{CK}:

\begin{thm}
\label{finiteres}
The complex $\widetilde{F'}$ of sheaves on $\X \times \X$ is a resolution of the diagonal.
\end{thm}

\begin{proof}[Proof of Theorem \ref{finiteres}]
Observe that the terms of the complex $F$ are positively graded. Let $F'' = F / F'$. The quotient $F''/(x_0, \dots, x_n)F''$ is a direct sum of twists of Koszul complexes on the $x_i$'s; it thus follows from Lemma \ref{homnak} and induction that $H_i(F'') = 0$ for $i > 0$. 
Thus, $H_i(F') = 0$ for $i > 0$ as well. As for $i = 0$: let $G^j$ be the subcomplex of $F''$ with $t^{\th}$ term
$$
G^j_t = \bigoplus_{\pi(a) \ge 0} \bigoplus_{\pi(b) = w-\pi(a)}^{w-\pi(a)+j} (\om_E)_{(b; t)} \otimes_k S'(-a-b,a).
$$
This is the sum of the first $j + 1$ full Koszul strands in the $x_i$'s. A direct calculation yields $H_0(\widetilde{G}^j) = 0$ for each $j$. Since $F''$ is the colimit of the $G^j$'s, we get $H_0(\widetilde{F'}) = H_0(\widetilde{F}) = \OO_\Delta$.
\end{proof}

\begin{remark}
\label{hirzres}
We believe that finite rank resolutions of the diagonal constructed as in Theorem~\ref{finiteres} exist more generally: see Conjecture~\ref{c:finite res} below. For instance, this has been verified in the case where $\X$ is a smooth projective toric variety of Picard rank 2 \cite{BSresolution}.
\end{remark}

%%%%%%%%%%%%%%
\section{Recovering a sheaf from its toric Tate resolution}
\label{BMsection}
%%%%%%%%%%%%%%

As discussed above, Eisenbud-Fl\o ystad-Schreyer explain in \cite{EFS} how to recover a sheaf on projective space from its Tate resolution \cite[Theorem 6.1]{EFS}. Our next goal is to generalize this result to toric varieties. In fact, we will prove the following stronger result. 

\begin{thm}
\label{BMdiagram}
Let $\X$ be a projective toric stack with Cox ring $S$, and let $E$ denote the Koszul dual exterior algebra of $S$. Let $\cR$ denote the resolution of the diagonal for $\X$ from Theorem \ref{resdiag} and $\K$ the kernel of the Fourier-Mukai transform in the definition of the toric Tate resolution functor (\S \ref{mainresults}). There exists a functor
$
\U\colon \DM(E) \to \Com(\X)
$
making
$$
\xymatrix{
& \DM(\X_E) \ar[r]^-{R\pi_{2*}}  & \DM(E) \ar[dd]^-{\U} \\
\Coh(\X) \ar[ru]^-{\pi_1^*(-) \otimes \K} \ar[rd]_-{\pi_1^*( - ) \otimes \cR} & &   \\
& \Com(\X \times \X) \ar[r]^-{R\pi_{2*}} & \Com(\X) \\
}
$$
commute up to isomorphism in $\Db(\X)$.
\end{thm}

Recall that the top row of the diagram in Theorem \ref{BMdiagram} gives the Tate resolution, up to homotopy equivalence. The following generalization of \cite[Theorem 6.1]{EFS} is therefore immediate from Theorems \ref{resdiag} and \ref{BMdiagram}: 

\begin{cor}
\label{monad}
Let $\F \in \Coh(\X)$. There is an isomorphism $\U\Tate(\F) \cong \F$ in $\Db(\X)$.
\end{cor}

\begin{proof}[Proof of Theorem \ref{BMdiagram}]
We begin by defining the functor $\U$. For $D \in \DM(E)$, we let $D'$ denote the submodule $\bigoplus_{-a \in V, i \in \Z} D_{(a;i)}$ of $D$, where $V$ is as defined in \S \ref{sec:res diagonal}.   We define $\U\colon \DM(E) \to \Com(\X)$ by sending $D$ to $\LL(D')\in \Com(S)$ and then sheafifying to obtain a complex of $\cO_\X$-modules.

Recall that $\Tate(\F)$ is homotopy equivalent to the differential module $Y$ that one gets by totalizing the bicomplex (\ref{cechbicomplex}) and then applying the equivalence $\Comper(E) \xra{\simeq} \DM(E)$ from Remark \ref{dgmodules}. It follows that there is a homotopy equivalence $\U\Tate(\F) \xra{\simeq} \tU(Y)$. The underlying $E$-module of $Y$ is $\bigoplus_{\l = 0}^{\dim X} \bigoplus_{a \in \Cl(X)}  \cC^{G,\l}_{\cF(a)} \otimes_k \om_E(-a; \l)$, and so $\tU(Y)_j$ is
$$
\bigoplus_{b \in V}  \bigoplus_{\l = 0}^{\dim X} \bigoplus_{a \in \Cl(X)} (\om_E)_{(-a -b, \l + j)}  \otimes_k \cC^{G,\l}_{\cF(a)} \otimes_k \OO(b) .
$$
Reindexing, we may write this sum as 
$$
\bigoplus_{a \in V}  \bigoplus_{\l = 0}^{\dim X} \bigoplus_{b \in \Cl(X)} (\om_E)_{(b, \l + j)}  \otimes_k \cC^{G,\l}_{\cF(-a-b)} \otimes_k \OO(a).
$$
The differential on $\tU(Y)$ sends a section $f \otimes c \otimes s$ of $(\om_E)_{(b, \l + j)} \otimes_k \cC^{G,\l}_{\cF(-a-b)} \otimes_k \OO(a)$ to
$$
\sum_{i = 0}^n (e_i f \otimes c \otimes x_is  + (-1)^{\l + 1} e_if \otimes y_ic \otimes s) - f \otimes \del_C(c) \otimes s,
$$
where $\del_C$ denotes the \v{C}ech differential. We can view $\U(Y)$ as the totalization of the bicomplex with $(p,q)$ entry $\bigoplus_{a \in V} \bigoplus_{b \in \Cl(X)} (\om_E)_{(b; p)} \otimes_k \cC^{G, -q}_{\cF(-a-b)} \otimes_k \OO(a)$, horizontal differential given by sending $f \otimes c \otimes s$ to $\sum_{i = 0}^n (e_i f \otimes c \otimes x_is + (-1)^{q + 1} e_if \otimes y_ic \otimes s)$, and vertical differential induced by $-\del_C$. For instance, notice that the $q = 0$ row is given by
applying $\cC^{G,0}$ to the first tensor factor of every term in $\pi_1^*(\F) \otimes \cR$, where $\cR$ denotes the resolution of the diagonal from Theorem \ref{resdiag}. By the projection formula and flat base change, the complex $\U(Y)$ is a model for the Fourier-Mukai transform $\Phi_{\cR}(\F) := R\pi_{2*}(\pi_1^*(\F) \otimes \cR)$ in $\Db(\X)$. Finally, observe that $\Phi_{\cR}(\F)$ is isomorphic to $\F$ in $\Db(\X)$. 
\end{proof}

\begin{question}
In~\cite[\S7]{EES} and \cite{Schrey-split}, an analogue of the functor $\U$ for products of projective spaces is used to give generalized Horrocks-type criteria for when a vector bundle splits as a sum of line bundles.  Do similar criteria hold for more general toric varieties? In a different direction: can $\U$ be used to study vector bundles on $\X$ in the manner of \cite{ES}?
\end{question}

A Horrocks splitting criterion in this vein for smooth projective toric varieties of Picard rank 2 is obtained in \cite{BSresolution}. 

Applying the functor $\U$ from the proof of Theorem \ref{BMdiagram} to the $E$-module $\om_E(d;i)$ gives the truncated and shifted Koszul complex $\widetilde{K}_d(d)[i]$. When $\X=\PP^n$, $\U( \om_E(d;0))$ is therefore quasi-isomorphic to $\Omega_{\PP^n}^d(d)$, and thus this functor is closely related\footnote{The functor $\U$ can behave differently over general projective toric stacks than over $\PP^n$. For instance, taking $\X$ to be the weighted projective stack $\PP(1,1,2)$, we have $\U(\om_E(1;0)) = (\OO(-1)^{\oplus 2} \xra{\left(\begin{smallmatrix}x_0 & x_1 \end{smallmatrix}\right)} \OO)$,
which has homology that is nonzero in both positions and is not given by a vector bundle in degree 0.} to the functor $\Omega$ defined in~\cite{EFS} and discussed in \S \ref{othertorics}. We now use this observation to make precise the rough intuition from \S \ref{othertorics} that, over $\PP^n$, applying $``\id \times \mathbf{\Omega}"$ to the complex \eqref{Kst} gives Beilinson's resolution of the diagonal. In fact, the following result gives an analogous statement over any projective toric stack; its proof is a straightforward calculation.

%Recall that $\K$ has underlying module $\bigoplus_{a \in \Cl(X)} \OO(a) \otimes_k \om_E(-a; 0)$ and differential given by left multiplication by $\sum_{i = 0}^n x_i \otimes e_i$.  
%\daniel{$e_i\otimes x_i$?} 
\begin{prop}
\label{1U}
Let $\U$ be the functor defined in the proof of Theorem \ref{BMdiagram}. The chain complex with $i^{\th}$ term $\bigoplus_{a \in \Cl(X)} \OO(a) \otimes_k \U(\om_E(-a; 0))_i$ and differential given by $\sum_{i = 0}^n x_i \otimes \U(e_i) - \id \otimes \del_{\U} $ is precisely the resolution of the diagonal from Theorem \ref{resdiag}.
\end{prop}

The idea is that the complex in Proposition \ref{1U} is obtained by applying ``$\id \times \U$" to the noncommutative Fourier-Mukai kernel $\K$ from Section \ref{mainresults}. 

%\begin{remark}
%The terms of the Beilinson monad $\U\Tate(\F)$ have infinite rank. 
%When $\X$ is a weighted projective stack, the definition of $\U$ can be modified so that $\U\Tate(\F)$ is quasi-isomorphic to the Fourier-Mukai transform $R\pi_{2 *}(\pi_1^*\F \otimes \cR')$, where $\cR'$ is the finite rank resolution of the diagonal from Theorem \ref{finiteres}. %In general, this same method can be applied to any toric stack $\X$ for which the resolution of the diagonal $\cR$ in Theorem \ref{resdiag} admits a finite rank subcomplex that is also a resolution (e.g. Example \ref{hirzres}).  This provides further motivation for Question~\ref{q:finite res}.
%\end{remark}

%%%%%%%%%%%%%%%%%%%%%%%%%%%%%%%%%%%
\section{The bounded derived category of a weighted projective stack}
\label{buchweighted}
%%%%%%%%%%%%%%%%%%%%%%%%%%%%%%%%%%%
 We recall that Tate resolutions over $\PP^n$ can be used to interpret $\Db(\PP^n)$ in terms of the exterior algebra, giving a geometric refinement of the classical BGG equivalence $\Db(S) \simeq \Db(E)$. In detail: let $K^{\on{ex}}(E)$ denote the homotopy category of (unbounded) exact complexes of finitely generated free $\Z$-graded $E$-modules. It follows from results of \cite{buchweitz} and \cite{EFS} that the Tate resolution functor implements an equivalence $\Db(\PP^n) \simeq K^{\on{ex}}(E)$. The goal of this section is to use our results on toric Tate resolutions to generalize this equivalence to weighted projective stacks.

Let us first fix some notation. Let $\X$ be a weighted projective stack with Cox ring $S$ and Koszul dual exterior algebra $E$. We will say an object $D \in \DM(E)$ is \defi{locally finite} if $\sum_{i \in \Z} \dim_k D_{(a;i)} < \infty$ for all $a \in \Cl(X)$. We let $K^{\on{ex}}_{\DM}(E)$ denote the homotopy category of exact, free, locally finite differential $E$-modules. We will prove:

\begin{thm}
\label{geobgg}
For a weighted projective stack $\X$, there is an equivalence
$
\Db(\X) \xra{\simeq} K^{\on{ex}}_{\DM}(E)
$
that sends a coherent sheaf concentrated in homological degree 0 to its Tate resolution.
\end{thm}

\begin{remark}
The results in this section all generalize in an evident way to generalized weighted projective stacks; we omit the details. \end{remark}

Suppose $\X$ is a weighted projective stack. Let $\Perf_{\DM}(E) \subseteq \Db_{\DM}(E)$ be the smallest triangulated subcategory of $\Db_{\DM}(E)$ containing all summands of finitely generated free flag differential $E$-modules, and define the \defi{singularity category of differential $E$-modules} to be the Verdier quotient
$
\on{D}^{\on{sing}}_{\DM}(E) := \Db_{\DM}(E) / \Perf_{\DM}(E).
$
We will prove Theorem \ref{geobgg} by constructing a chain of equivalences
\begin{equation}
\label{chain}
\Db(\X) \xra{\simeq} \on{D}^{\on{sing}}_{\DM}(E)  \xra{\simeq} K^{\on{ex}}_{\DM}(E).
\end{equation}
The second equivalence is an analogue of a theorem of Buchweitz (\cite[Theorem 4.4.1]{buchweitz}).

We start with the first link in the chain. We say a graded $S$-module is \defi{torsion} if it is annihilated by a power of the maximal ideal $\langle x_0, \dots, x_n \rangle$. Denote by $\Db_{\on{tors}}(S)$ the subcategory of $\Db(S)$ given by complexes with torsion homology.

\begin{prop}
\label{singeq}
The equivalence $\on{D}^{\on{b}}_{\DM}(E) \simeq  \Db(S)$ of Corollary \ref{cor:bounded} induces equivalences
$$
 \Db_{\on{tors}}(S) \simeq \Perf_{\DM}(E)
\qquad
and 
\qquad
\Db(\X) \xra{\simeq} \on{D}^{\on{sing}}_{\DM}(E) .
$$
\end{prop}

\begin{proof}[Proof of Proposition~\ref{singeq}]
We recall that a triangulated subcategory $\mathcal{T'}$ of a triangulated category $\mathcal{T}$ is called \defi{thick} if $\mathcal{T'}$ is closed under taking summands. Since a finitely generated free flag is precisely the same thing as a finite iterated extension of rank 1 free modules, one concludes that the subcategory $\Perf_{\DM}(E) \subseteq \Db_{\DM}(E)$ is the thick subcategory generated by $E(a; 0)$ for all $a \in \Z$. The first equivalence follows immediately, since $\Db_{\on{tors}}(S)$ is the thick subcategory of $\Db(S)$ generated by $k(a)$ for all $a \in \Z$, and $\RR(k(a)) = E(a; 0)$. 

As for the second equivalence, it suffices to observe that $\Db(\X)$ is equivalent to the Verdier quotient $\Db(S) / \Db_{\on{tors}}(S)$. To see this, apply  \cite[Proposition 2.17]{orlov} to see that $\on{Coh}(\X)$ is equivalent to the abelian quotient of the category of finitely generated graded $S$-modules by the Serre subcategory given by torsion modules. Then recall that, for any abelian category $\mathcal{A}$ and Serre subcategory $\mathcal{S}$ of $\mathcal{A}$, there is an equivalence $\Db(\mathcal{A}) / \Db_{\mathcal{S}}(\mathcal{A}) \simeq \Db(\mathcal{A}/\mathcal{S})$, where 
$
\Db_{\mathcal{S}}(\mathcal{A}) = \{X \in \Db(\mathcal{A}) \text{ : } H_n(X) \in \mathcal{S} \text{ for all } n \in \Z\}.
$ 
\end{proof}

Any object in $\DM(E)$ with finitely generated homology admits a locally finite free resolution, by Theorem \ref{EU}. Let $D$ be a finitely generated object in $\DM(E)$, $F$ a locally finite free resolution of $D$, and $G$ a locally finite free resolution of $D^\vee  := \underline{\Hom}_E(D, E)$. Since dualizing over $E$ is exact, and every finitely generated module over $E$ is maximal Cohen-Macaulay, we have an induced quasi-isomorphism
$
F \xra{\simeq} D \cong (D^\vee)^\vee \xra{\simeq} G^\vee.
$
Denote the mapping cone of this quasi-isomorphism by $\bC(D)$. The object $\bC(D)$ is contained in $K^{\on{ex}}_{\DM}(E)$, and it is well-defined up to the choices of $F$ and $G$. If $f : D \to D'$ is a quasi-isomorphism in $\DM(E)$, there is an induced homotopy equivalence $\bC(D) \xra{\simeq} \bC(D')$, and so there is an induced functor $\bC\colon \Db_{\DM}(E) \to K^{\on{ex}}_{\DM}(E)$ (recall from Proposition \ref{fg} that every object in $\Db_{\DM}(E)$ is isomorphic to a finitely generated differential module). Moreover, if $D \in \Perf_{\DM}(E)$, then $D$ and $D^\vee$ are free resolutions of themselves, so $C(D) = \cone(D \xra{\cong} (D^\vee)^\vee)$, which is contractible. It follows that we have an induced functor
$
\bC\colon \on{D}^{\on{sing}}_{\DM}(E) \to K^{\on{ex}}_{\DM}(E).
$
The definition of the map $\bC$ closely resembles Buchweitz's construction of \defi{complete resolutions} in \cite{buchweitz}; hence the notation ``$\bC$" for this functor.

Going the other direction, let $D \in K^{\on{ex}}_{\DM}(E)$, and write $D = \bigoplus_{(a; i) \in \Z \times \Z} E(a; i)^{\oplus r_{a,i}}$. For all $a \in \Z$, write $D_a = \bigoplus_{i \in \Z} E(a; i)^{\oplus r_{a,i}}$. Let $D' = \bigoplus_{a \le 0} D_a$ and $D'' = \bigoplus_{a > 0} D_a$. Decomposing $D$ as $D' \oplus D''$, we can write the differential $\del$ on $D$ as a matrix of the form $\begin{pmatrix} \del' & \alpha \\ 0 & \beta \end{pmatrix}$. Notice that $\alpha$ exhibits $(D''(0;1), -\beta)$ as a free resolution of $(D', \del')$, and $D = \cone(\alpha)$. Observe also that, since $\dim_k \bigoplus_{i \in \Z} D_{(a; i)} < \infty$ for all $a \in \Z$, the differential module $D'$ has finitely generated homology. We thus have a functor
$
\tau_{\le 0} : K^{\on{ex}}_{\DM}(E) \to \on{D}^{\on{sing}}_{\DM}(E)
$
given by $D \mapsto D'$.

\begin{prop}
\label{othereq}
The functors $ \bC\colon \on{D}^{\on{sing}}_{\DM}(E) \rightleftarrows  K^{\on{ex}}_{\DM}(E)  :\tau_{\le 0}  $ are inverse equivalences.
\end{prop}
\begin{proof}

Let $D \in \on{D}^{\on{sing}}_{\DM}(E)$ be a finitely generated differential module. Let $F$ be a locally finite free resolution of $D$ and $G$ a locally finite free resolution of $D^\vee$, so that $\bC(D) = F(0;-1) \oplus G^\vee$. Write $\bC(D) = \bigoplus_{(a; i) \in \Z \times \Z} E(a; i)^{\oplus r_{a,i}}$ and $\bC(D)_a =  \bigoplus_{i \in \Z} E(a; i)^{\oplus r_{a,i}}$, and choose $a \ll 0$ such that $\bC(D)_b \subseteq F$ for all $b \le a$. Let $N = \bigoplus_{b \le a} \bC(D)_b$. The natural maps
$$
N \to (\tau_{\le 0} \circ \bC)(D) \quad \text{and} \quad N \to F
$$
are both isomorphisms in $\on{D}^{\on{sing}}_{\DM}(E)$; it follows that $(\tau_{\le 0} \circ \bC)(D) \cong D$ in $\on{D}^{\on{sing}}_{\DM}(E)$.

On the other hand, let $D \in K^{\on{ex}}_{\DM}(E)$, and let $D'$ and $D''$ be as in the above construction of the functor $\tau_{\le 0}$. The differential module $D'$ has finitely generated homology: let $Y$ be the $E$-submodule of $D'$ generated by elements of $\Cl(X)$-degree $a$ such that $H(D')_{(a; i)} \ne 0$ for some $i \in \Z$. The object $Y$ is finitely generated, and it is a differential submodule of $D'$, because the differential is degree 0 with respect to the $\Cl(X)$-degree. Moreover, the inclusion $Y \into D'$ is a quasi-isomorphism; in fact, the dual $(D')^\vee \to Y^\vee$ of this inclusion is a free resolution of $Y^\vee$. Choose locally finite free resolutions $F$ of $Y$ and $G$ of $Y^\vee$, so that $(\bC \circ \tau_{\le 0})(D)$ is isomorphic to $\cone(F \xra{\simeq} G^\vee)$. Since $G^\vee$ is homotopy equivalent to $D'$ and $F$ is homotopy equivalent to $(D''(0;1),-\b)$, $D$ is homotopy equivalent to $(\bC \circ \tau_{\le 0})(D)$.
\end{proof}

\begin{proof}[Proof of Theorem \ref{geobgg}]
Our equivalence is given by the composition
$
\Db(\X) \xra{\RR} \on{D}^{\on{sing}}_{\DM}(E) \xra{\bC} K^{\on{ex}}_{\DM}(E) 
$
of the equivalences in Propositions \ref{singeq} and \ref{othereq}. By Theorem \ref{weightedthm}, applying this equivalence to a sheaf $\F$ concentrated in degree 0 gives the Tate resolution of $\F$. 
\end{proof}

%%%%%%%%%%%%%%%%%%%%%%%%%
\section{Future directions}\label{sec:applications}

A running theme in this work, and elsewhere~\cite{BES,BHS, BKLY,costa-miro-roig,EE, EES,Ha1,Ha2,hering, HSS, MS,  MS2,yang-monomial}, is that the multigradings on the Cox rings of toric varieties make many homological constructions more subtle than their well-known counterparts over $\PP^n$. We propose a few questions in this vein.  
%For instance, Theorems \ref{thm:toric exact intro} and \ref{thm:exactness properties} motivate the following 

\begin{conj}\label{c:intrinsic}
The properties in Theorem \ref{thm:toric exact} characterize $\Tate(\cF)$ up to isomorphism of differential modules.
\end{conj}

In general, we have only shown that $\Tate(\cF)$ is well-defined up to homotopy, but Theorems~\ref{weightedthm} and \ref{comparison} imply that Conjecture \ref{c:intrinsic} holds over generalized weighted projective stacks and products of projective spaces. See the paragraph beneath the proof of Theorem \ref{comparison} for a discussion of the difficulty of extending these results to general projective toric stacks.

\begin{conj}\label{c:finite res}
The resolution $F$ in Theorem \ref{resdiag} always admits a finite rank subcomplex whose sheafification is a resolution of the diagonal. In fact, we can always find such a subcomplex whose length is at most $\dim(X)$.
\end{conj}
Conjecture \ref{c:finite res} would resolve a question of Berkesch--Erman--Smith on the minimal length of a virtual resolution of a module (see~\cite[Question~6.5]{BES}), and it would imply a large swath of new cases of a Conjecture of Orlov concerning the Rouquier dimension of the bounded derived category of a quasi-projective variety \cite[Conjecture 10]{orlovremarks}. Conjecture~\ref{c:finite res} has been proven by the first author and Sayrafi in the case of smooth projective toric varieties of Picard rank 2 \cite{BSresolution}.

One application of the Tate resolutions from~\cite{EFS} was to the development of an efficient algorithm for computing sheaf cohomology on $\PP^n$; see also~\cite{decker-eisenbud}.  
\begin{question}\label{q:algorithm}
Can one use toric Tate resolutions to develop an exterior algebra algorithm for computing sheaf cohomology on any projective toric stack? 
\end{question}
In followup work, we will show that this question has a positive answer for weighted projective stacks, using the theory of minimal free resolutions of differential modules from~\cite{BE}. 
%Since Theorem \ref{weightedthm} does not immediately extend to more general toric stacks, this algorithm does not either. However, an algorithm for computing cohomology over products of projective spaces using the Tate resolution is developed in \cite{EES}; this algorithm uses that the Tate resolution, while not itself the mapping cone of a minimal free resolution of $\RR(M)$ for some module $M$, can nevertheless be described in terms of minimal free resolutions in this case. We expect that a description of toric Tate resolutions involving minimal free resolutions in general will guide the way to a generalization of Algorithm \ref{algorithm}. 
Just as with the problem of characterizing the toric Tate resolution up to isomorphism (Conjecture~\ref{c:intrinsic}), we expect that extending such an algorithm to a more general toric variety will rely heavily upon the exactness properties in Theorem \ref{thm:toric exact}(4).
%
%

%o relate the derived category of coherent sheaves on $\PP^n$ to the singularity category of $E$ and to the homotopy category of totally acyclic complexes on $E$.  \daniel{We must have some old boilerplate introduction mentioning this for $\PP^n$?  I feel like we are missing a sentence here.}
%
%\begin{question}\label{q:intrinsic}
%Do the properties in Theorems \ref{thm:toric exact intro} and \ref{thm:exactness properties}  characterize $\Tate(\cF)$ up to isomorphism of differential modules? 
%\end{question}
%
%\begin{question}\label{q:finite res}
%Does the resolution $F$ in Theorem \ref{introres} always admit a finite rank subcomplex whose sheafification is a resolution of the diagonal? Can we choose such a subcomplex whose length is at most $\dim(X)$? 
%\end{question}

\begin{question}\label{q:derived}
Can one generalize Theorem \ref{geobgg} above by relating the bounded derived category of a projective toric stack $\X$ to an appropriate homotopy category of exact differential $E$-modules?
\end{question}

A positive answer to Question \ref{q:derived} could provide new connections to the study of derived categories of toric varieties and stacks, e.g.~\cite{bdf,bdfkk,bfk,BH,cmr,craw-smith,dlm,efimov,k1,k2,k3,orlov}. For instance, an exceptional object/collection on the $E$-module side would immediately yield the same for $\Db(\X)$. 

A fundamental challenge underlying Question \ref{q:derived} is determining the correct analogue of $K^{\on{ex}}_{\DM}(E)$ in the case of an arbitrary projective toric stack $\X$. 
To explain the difficulty: notice that passing from $\Db(S)$ to $\Db(\X)$ requires one to take a quotient not just by complexes supported in the maximal ideal $\langle x_0, \dots, x_n \rangle$, but complexes supported in the irrelevant ideal. The counterpart of $\Db(\X)$ will therefore typically be strictly smaller than $K^{\on{ex}}_{\DM}(E)$. One can see this from another point of view: by Theorem \ref{thm:toric exact}(3), the Tate resolution satisfies more subtle exactness properties than a typical object in $K^{\on{ex}}_{\DM}(E)$, and so we should not expect every object in $K^{\on{ex}}_{\DM}(E)$ to correspond to an object in $\Db(\X)$ via the Tate resolution functor.

%%%%%%%%%%%%%%%%%%%%%%%%%%%%%%%
\appendix

%%%%%%%%%%%%%%%%%%%%%%%%%%%%%%%
\section{Positive multigradings}
\label{sec:positivegrading}
%%%%%%%%%%%%%%%%%%%%%%%%%%%%%%%
Let $A$ be an abelian group and $R$ an $A$-graded ring with $R_0$ a field. 
\begin{defn}
\label{posring}
A group homomorphism $\theta\colon A \to \Z$ is called a \defi{positive $A$-grading} on $R$ if, for all $x\in R\setminus\{0\}$:
\begin{enumerate}
\item $\theta(\deg(x)) \ge 0$, and
\item $\theta(\deg(x)) = 0$ if and only if $x$ is a unit.
\end{enumerate}
The ring $R$ is called \defi{positively $A$-graded} if a positive $A$-grading exists. The above properties ensure that $R$ is, via $\theta$, a nonnegatively $\Z$-graded ring such that the ideal in $R$ generated by elements of positive degree is maximal.
\end{defn}

%Any homomorphism $\theta : A \to \Z$ determines a $\Z$-grading of $R$ by setting $R_i = \bigoplus_{\theta(a) = i} R_a$. When $\theta$ is a positive $A$-grading, this $\Z$-grading is nonnegative. 
%
\begin{example}
\label{cox}
Let $X$ be a projective toric variety with $\Cl(X)$-graded Cox ring $S$. We claim that $S$ is positively $\Cl(X)$-graded. Indeed, letting $H$ be an ample divisor on $X$, the map
$
\theta \colon \Cl(X) \to \Z
$
given by intersecting with $H^{\dim X - 1}$ is a positive $\Cl(X)$-grading. The projectivity assumption cannot be removed; see e.g. \cite[Example 5.2.3]{CLS}.
\end{example}

\begin{defn}
\label{posmod}
Let $M$ be an $A$-graded $R$-module. A homomorphism $\theta : A \to \Z$ determines a $\Z$-grading on $M$ by setting 
$M_i = \bigoplus_{\theta(a) = i} M_a$. We say $M$ is \defi{positively $A$-graded} if there is a positive $A$-grading on $R$ such that this associated $\Z$-grading of $M$ is bounded below. 
\end{defn}

%Let $R_+ = \bigoplus_{a \ne 0} R_a$. When $R$ is positively graded, $R_+$ is an ideal. 

We record the following homological variant of Nakayama's Lemma:

\begin{lemma}
\label{homnak}
Let $R$ be positively $A$-graded, and let $C$ be a complex of free $R$-modules such that $H_i(C)$ is positively $A$-graded. Let $x \in R$ be  homogeneous of nonzero degree. If $H_i(C/xC) = 0$, then $H_i(C) = 0$.
\end{lemma}
%
%\begin{proof}
%Let $K$ be the Koszul complex on $x$. Consider the bicomplex $B$ with
%$$
%B_{p,q} = \begin{cases} C_q, & p = 0 \\ 
%C_q(-\deg(x)), & p = 1 \\
%0, & \on{else;} \end{cases}
%$$
%with horizontal differential given by multiplication by $x$ and vertical differential given by the differential on $C$ multiplied by $(-1)^p$. There is a spectral sequence
%$$
%E^2_{p,q} = H^h_p H^v_q(C) \Rightarrow H_{p + q}(C \otimes K) \cong H_{p+q}(C/xC).
%$$
%Since $E^2_{p,q} = 0$ when $p \ne 0,1$, we have short exact sequences
%$$
%0 \to E^2_{1, n-1} \to H_n(C/xC) \to E^2_{0,n} \to 0
%$$
%for each $n$. Noting that $E^2_{0, n} = H_n(C)/xH_n(C)$, we conclude that $H_i(C)/xH_i(C) = 0$. Now apply (the graded version of) Nakayama's Lemma.\end{proof}

%%%%%%%%%%%%%%%%%%%%
\section{Further background on differential $E$-modules}
\label{DM}
%%%%%%%%%%%%%%%%%%%%

As in \S \ref{sec:BGGsection}, let $S = k[x_0, \dots, x_n]$ be positively graded by an abelian group $A$ and let $E = \Lambda_k(e_0, \dots, e_n)$ be equipped with the $A \oplus \Z$-grading given by $\deg(e_i) = (-\deg(x_i); -1)$.

%%%%%%%%%%%%%%%%%%%%%%%%%%%
\subsection{Resolutions of differential $E$-modules}
\label{freeres}
%%%%%%%%%%%%%%%%%%%%%%%%%%%%

We recall from \cite{BE} the notion of a free resolution of a differential module. 

\begin{defn}[\cite{BE} Section 1]
\label{freeres}
A differential $E$-module $F$ is a \defi{free flag} if $F$ is a free module that may be equipped with a decomposition $\bigoplus_{i \ge 0} F_i$ such that $\del_F(F_i) \subseteq \bigoplus_{j < i} F_j$. Given $D \in \DM(E)$, a \defi{free flag resolution} of $D$ is a quasi-isomorphism $F \xra{\simeq} D$, where $F$ is a free flag. A \defi{free resolution} of $D$ is a quasi-isomorphism $F \xra{\simeq} D$ that factors as $F \to \widetilde{F} \to D$, where $\widetilde{F} \to D$ is a free flag resolution, and $F \to \widetilde{F}$ is a split injection. We say a free resolution $F \xra{\simeq} D$ is \defi{minimal} if $\del_F(F) \subseteq \m F$, where $\m = \langle e_0, \dots, e_n \rangle \subseteq E$.

Reversing arrows, one can define injective resolutions. A differential $E$-module $I$ is an \defi{injective coflag} if $I$ is an injective module that may be equipped with a decomposition $\bigoplus_{i \le 0} I_i$ such that $\del_I(I_i) \subseteq \bigoplus_{ j < i} I_j$ (note that graded injective, projective, and free modules coincide over $E$). One defines an \defi{injective coflag resolution} in the evident way, and an \defi{injective resolution} of a differential module $D$ is a quasi-isomorphism $D \xra{\simeq} I$ that factors as $D \to \widetilde{I} \to I$, where $D \to \widetilde{I}$ is an injective coflag resolution, and $\widetilde{I} \to I$ is a split surjection.
\end{defn}

The following fact plays a key role in the proof of Theorem \ref{weightedthm} above:

\begin{thm}
\label{EU}
Any differential $E$-module $D$ whose homology is finitely generated admits a minimal free resolution $F \xra{\simeq} D$, and this minimal free resolution is unique up to isomorphism of differential modules. Moreover, we have $ \sum_{i \in \Z} \dim_k F_{(a, i)} <\infty$ for all $a \in \Cl(X)$, and $F$ is positively $A$-graded, in the sense of Definition \ref{posmod}.
\end{thm}

\begin{proof}
By \cite[Theorem 1.2]{BE}, minimal free resolutions exist and are unique for any differential module with finitely generated homology and degree 0 differential over a (possibly noncommutative) $\Z$-graded local ring $R$ such that $R_0$ is a field. Our result does not immediately follow from this theorem, because $E$ is $A \oplus \Z$-graded, not $\Z$-graded, and the differential on $D$ has degree $(0, -1)$, not $0$. However, a slight modification of the arguments in \cite{BE} gives the result we want. In detail: the positive $A$-grading on $S$ induces a positive $A$-grading on $E$ in the evident way. We can use this induced $\Z$-grading to construct a minimal free resolution of $D$ exactly as in \cite[Remark 5.7]{BE}; this proves existence. Uniqueness follows from an argument identical to the proof of the uniqueness part of \cite[Theorem 4.2(b)]{BE}.
\end{proof}

\iffalse
\begin{proof}
By \cite[Theorem 1.2]{BE}, minimal free resolutions exist and are unique for any differential module with finitely generated homology and degree 0 differential over a (possibly noncommutative) $\Z$-graded local ring $R$ such that $R_0$ is a field. Our result does not immediately follow from this theorem, because $E$ is $A \oplus \Z$-graded, not $\Z$-graded, and the differential on $D$ has degree $(0, -1)$, not $0$. But, our result does follow immediately from the proof of this theorem. In detail: by Example \ref{cox}, we may choose a positive $A$-grading on the Cox ring $S$ of $\X$ (Definition \ref{posring}); this induces a positive $A$-grading on $E$ in an obvious way. We can construct an $A\times \Z$-graded free flag resolution $\widetilde{F}$ of $D$ using the method of \cite[Theorem 3.2]{BE}; since $E$ is positively graded, it follows immediately from the construction that $\widetilde{F}$ is positively graded (Definition \ref{posmod}), and $\dim_k \widetilde{F}_a < \infty$ for all $a \in A$. Our result now follows by an argument identical to the proof of \cite[Theorem 4.2(b)]{BE}.
\end{proof}
\fi

%%%%%%%%%%%%%%%%%%%%%%%%%
\subsection{Tensor product and internal $\Hom$ for differential $E$-modules}
\label{tensorhom}
%%%%%%%%%%%%%%%%%%%%%%%%%

We can use the auxiliary $\Z$-grading on $E$ to define a tensor product, internal Hom, $\Tor$, and $\Ext$ for differential $E$-modules; these coincide with the usual notions for dg-modules via the first equivalence discussed in Remark \ref{dgmodules}. Letting $D, D' \in \DM(E)$, we define the tensor product $D \otimes_E^{\DM} D'$ to be the differential module with underlying module $D \otimes_E D'$ and differential
$$
d \otimes d' \mapsto \del_D(d) \otimes d' + (-1)^{\aux(d)}d \otimes \del_{D'}(d'),
$$
where, as stated in Conventions \ref{leftconvention}, $\aux( - )$ denotes the auxiliary degree. Recall also from Conventions \ref{leftconvention} that any right $E$-module may be considered as a left $E$-module in a canonical way, so the tensor product $D \otimes_E D'$ makes sense. The internal Hom object $\underline{\Hom}_E^{\DM}(D, D')$ is defined to be the differential module with underlying $E$-module $\underline{\Hom}_E(D, D')$ and differential 
$
f \mapsto \del_{D'} \circ f  - (-1)^{\on{aux}(f)} f \circ \del_D.
$
Let $F$ be a free resolution of $D$ (Definition \ref{freeres}). We define
$$
\Tor^E_{\DM}(D, D') = H(F \otimes^{\DM}_E D') \quad \text{and} \quad
\Ext_E^{\DM}(D, D') = H(\underline{\Hom}^{\DM}_E(F, D')).
$$
One can also define $\Tor$ (resp $\Ext$) using a free (resp. injective) resolution of $D'$.

%\daniel{Make sure we're referencing: Farkas/Aprodu/Raicu group?; Ein-Lazarsfeld?. Borisov!  What's the latest Borisov paper, there must be a connection?  Gotta' add this. Also go through bib file and see if there are references we deleted but want to bring back.}
%The following is easily verified:
%\begin{prop}
%\label{foldunfold}
%The functors $\Fold$ and $\on{Unfold}$ are inverse equivalences.
%\end{prop}

\bibliographystyle{amsalpha}
\bibliography{Bibliography}

\end{document}